\documentclass[12pt,reqno]{amsart}

\usepackage{amsmath,amssymb,latexsym,soul,cite,mathrsfs}
\usepackage{mathtools}
\usepackage{ esint }
\usepackage{enumerate}
\usepackage{lipsum}

\usepackage{color,graphicx}
\usepackage[colorlinks=true,urlcolor=blue,
citecolor=red,linkcolor=blue,linktocpage,pdfpagelabels,
bookmarksnumbered,bookmarksopen]{hyperref}
\usepackage[english]{babel}

\usepackage[left=2.5cm,right=2.5cm,top=2.5cm,bottom=2.5cm]{geometry}

\usepackage[colorinlistoftodos]{todonotes}
\usepackage{dsfont}
\makeatletter
\providecommand\@dotsep{5}
\def\listtodoname{List of Todos}
\def\listoftodos{\@starttoc{tdo}\listtodoname}
\makeatother

\numberwithin{equation}{section}

\newtheorem{prop}{Proposition}[section]
\newtheorem{lemma}{Lemma}[section]
\newtheorem{theorem}{Theorem}[section]
\newtheorem{remark}{Remark}[section]

\numberwithin{equation}{section}

\theoremstyle{definition}
\newtheorem{definition}{Definition}[section]

\begin{document}
	
	\title{A two-phase free boundary problem involving exponential operator}

	\author{Pedro F. Silva Pontes}
	\author{Minbo Yang}

	\address[Pedro Fellype Silva Pontes]
	{\newline\indent Zhejiang Normal University
		\newline\indent
		School of Mathematical Sciences
		\newline\indent
		Jinhua 321004 -- People's Republic of China}
	\email{\href{fellype.pontes@gmail.com}{fellype.pontes@gmail.com}}
	
	\address[Minbo Yang]
	{\newline\indent Zhejiang Normal University
		\newline\indent
		School of Mathematical Sciences
		\newline\indent
		Jinhua 321004 -- People's Republic of China}
	\email{\href{mbyang@zjnu.edu.cn}{mbyang@zjnu.edu.cn}}

	\pretolerance10000
	
	
	\begin{abstract}
	In this paper we are interested in the study  of a two-phase problem equipped with the $\Phi$-Laplacian operator 
		$$	\Delta_\Phi u \coloneqq \mbox{div} \left(\varphi(|\nabla u|)\dfrac{\nabla u}{|\nabla u|}\right),$$
	where $\Phi(s)=e^{s^2}-1$ and $\varphi=\Phi'$. We obtain the existence, boundedness, and Log-Lipschitz regularity of the minimizers of the energy functional associated to the two-phase problem. Furthermore, we also prove that the free boundaries of these minimizers have locally finite perimeter and Hausdorff dimension at most $(N-1)$.
	\end{abstract}

\thanks{Pedro Fellype Silva Pontes was partially supported by ZJNU (YS304024910) and Minbo Yang, who is the corresponding author, was partially supported by the National Key Research and Development Program of China (No. 2022YFA1005700), National Natural Science Foundation of China (12471114) and Natural Science Foundation of Zhejiang Province (LZ22A010001).}

\maketitle

	\section{Introduction}
In the last decades, there has been a great increase in interest in the study of two-phase partial differential equations arising from multiphase flow in porous media, phase transmission in materials science and the dynamics of biological systems. Important examples of two-phase problems include multiphase flow in porous media, where understanding the behavior of different fluids (e.g., water, oil, gas) in porous substrates such as soil and rock is of great importance in fields such as petroleum engineering, hydrology, and soil mechanics, see for example \cite{AMW,BB}. Furthermore, in materials science, phenomena such as the famous ice-water transition illustrate the importance of two-phase models, see \cite{AT,S}. In addition, biological systems and biomedical engineering, as well as environmental and atmospheric sciences, also feature two-phase dynamics, exemplified by phenomena such as ocean-atmosphere interactions and multiphase reactions in air pollution, see \cite{ACM,BV,SP} and references therein.

There is a lot of literature dealing with two-phase problems, for instance, Amaral and Teixeira \cite{AT},  Braga \cite{B}, Braga and Moreira \cite{BM}, Leit\~{a}o, de Queiroz and Teixeira \cite{LQO}, Shrivastava \cite{S}, Zheng, Zhang and Zhao \cite{ZZZ}. Specifically, the investigations conducted in \cite{AT} and \cite{S} focused on the analysis of minimizers associated with the following functional:
$$J(\Omega,v) = \int_\Omega \Big(\mathcal{A}(x,v)|\nabla v|^p - f(x,v)v + \gamma(x,v)\Big)dx,$$
where
$$\mathcal{A}(x,s) = \mathcal{A}_+(x)\chi_{[s>0]} + \mathcal{A}_-(x)\chi_{[s\le 0]},$$
$$f(x,s) = f_+(x)\chi_{[s>0]} + f_-(x)\chi_{[s\le0]} \; \mbox{and} \; \gamma(x,s) = \gamma_+(x)\chi_{[s>0]} + \gamma_-(x)\chi_{[s\le0]},$$
with $\lambda \le \mathcal{A}_\pm(x) \le \Lambda$, for some $0<\lambda \le \Lambda <\infty$ fixed constants, $\gamma_\pm$ are continuous and integrable real valued functions on $\Omega$, $f_\pm \in L^q(\Omega)$ for $q > \frac{N}{p}$ and $p=2$, for \cite{AT}, or $p\ge2$, for \cite{S}. Their respective studies have demonstrated the existence, along with H\"{o}lder regularity, of minimizers for the functional $J$. Furthermore, with regard to the properties established for free boundary, \cite{AT} established the geometric non-degeneracy of the free boundary $\partial\{u>0\}$, while \cite{S} proved that the free boundaries $\partial \{u>0\}$ and $\partial \{u<0\}$ have locally finite perimeter.

It is natural to consider the two-phase problems with operators much more comprehensive than the operator $\mathcal{A}$. One generalization is to consider a function $G$ satisfying, for some $\delta,g_0>0$,
\begin{equation}\label{Lie}
	\delta\le \dfrac{g'(s)s}{g(s)} \le g_0,
\end{equation}
where $g=G'$, and then the $G$-Laplacian operator, defined by
\begin{equation}\label{G-lapla}
	\Delta_G u \coloneqq \mbox{div} \left(g(|\nabla u|)\dfrac{\nabla u}{|\nabla u|}\right).
\end{equation}
As is already well known, Lieberman \cite{L} proved that under this condition, the $G$-Laplacian operator is uniformly elliptic, and conversely, the ellipticity condition implies this inequality (see \eqref{UE}). This leads to ask if one can consider the problem equipped with operators that do not satisfy such kind of ellipticity condition.	An operator that has been highlighted in recent years is the $\Phi$-Laplacian operator, defined by \eqref{G-lapla} replacing $G$ by $\Phi$, with $\mathcal{N}$-function $\Phi(s) = e^{s^2}-1$, see Section \ref{SecOS} for the definition of $\mathcal{N}$-function. We would also like to mention, for example, the works of Bocea and Mih\u{a}ilescu \cite{BM}, Duc and Eells \cite{DE}, Naito \cite{N}, Santos and Soares \cite{SS,SS2}, and Zhang and Zhou \cite{ZZ}, in which the $\Phi$ function was considered.
In a recent paper \cite{SS} by Santos and Soares, the authors focused on the one-phase problem
$$I(u) = \int_\Omega \Big(\Phi(|\nabla u|) +\lambda\chi_{[u>0]}\Big)dx,$$
for some $\lambda>0$. They proved the existence and Lipschitz regularity of minimizers of the functional $I$, as well as some properties of the free boundary $\partial \{u>0\}$, such as density, porosity and Hausdorff dimension.

In the present paper, we are going to investigate the following functional:
$$\mathcal{J}_{\Phi,f,\gamma}(u) = \int_\Omega \Big(\Phi(|\nabla u|) - f(x,u)u + \gamma(x,u)\Big)dx,$$
where $\Omega\subset \mathbb{R}^N$, with $N\ge 2$, is a bounded domain with smooth boundary, and
$$
\begin{array}{c}
	\Phi(s)=e^{s^2}-1,\\
	f(x,s) = f_+(x)\chi_{[s>0]} + f_-(x)\chi_{[s\le0]} \quad \mbox{and} \quad \gamma(x,s) = \gamma_+(x)\chi_{[s>0]} + \gamma_-(x)\chi_{[s\le0]},
\end{array}
$$
where $f_\pm \in L^\infty(\Omega)$ and $\gamma_\pm$ are continuous and integrable real valued functions on $\Omega$. Actually, considering $K_\Phi(\Omega)$ as the Orlicz class and a prescribed function $\psi \in C(\overline{\Omega})$, with $|\nabla \psi| \in K_\Phi(\Omega)$, we set
$$W_\psi^{1,\Phi}(\Omega) = \Big\{ u \in W^{1,\Phi}(\Omega) \ : \ |\nabla u| \in K_\Phi(\Omega) \ \mbox{and} \ u-\psi \in W_0^{1,\Phi}(\Omega) \Big\}.$$ 
and consider the following minimizing problem
\begin{equation}\label{mainproblem}
	\min \Big\{ \mathcal{J}_{\Phi,f,\gamma}(u) \ : \ u \in W_\psi^{1,\Phi}(\Omega) \Big\}.
\end{equation}

	Before presenting the main results of our paper, we would like to point out that this problem can be applied to population dynamics in the following way. The dispersal of individuals is influenced by population density gradients, reflecting their search for favorable survival conditions. To describe scenarios in which the response to density variation is highly sensitive, we employ the exponential operator $\Phi$. This operator captures nonlinear behaviors, in which small increases in gradients can result in exponential dispersal rates. These phenomena are common in heterogeneous environments, such as fragmented ecosystems, in which populations avoid overcrowded or high-risk areas. See, for example, \cite{ADS,CC,H}.
	
	In the functional associated with the problem, $\mathcal{J}_{\Phi,f,\gamma}(u)$, the term $f(x,s)$ describes the growth or mortality rates of the population, which depend on both the local population density and the environmental characteristics. On the other hand, $\gamma(x,s)$ models external interactions, such as resource availability or predator presence, capturing distinct dynamics in regions where population density is positive or negative. This flexible structure enables the description of the coexistence of favorable and unfavorable zones for survival.
	
	Moreover, the exponential operator and the functions $f(x,s)$ and $\gamma(x, s)$ play crucial roles in the analysis of solution regularity. Together, they allow for the characterization of free boundaries associated with regions of distinct population density phases and the identification of critical transitions between habitats. This type of modeling is particularly relevant in practical applications, such as studies of invasive species, conservation planning, and sustainable ecosystem management, providing a detailed view of ecological interactions in complex systems.

Now we are able to state the main results. First of all, we have the existence of minimizers of $\mathcal{J}_{\Phi,f,\gamma}$, as well as the boundedness in the $L^\infty$-norm as follows:
\begin{theorem}\label{existenceintro}
	There exists a minimizer $u_0 \in W_\psi^{1,\Phi}(\Omega)$ of $\mathcal{J}_{\Phi,f,\gamma}$. Furthermore, $u_0 \in L^\infty(\Omega)$ and there exists $C>0$ which depends on $\displaystyle \|f\|_\infty,\mathcal{L}^N(\Omega),N,\sup_{\overline{\Omega}}\psi$, such that
	$$\|u_0\|_\infty \le C.$$
\end{theorem}

After establishing the existence and the $L^\infty$-boundedness of the minimizers, we will continue to investigate the regularity. More precisely, we prove the Log-Lipschitz regularity in the following result.

\begin{theorem}\label{reg}
	Let $u_0 \in W_\psi^{1,\Phi}(\Omega)$ be the minimizer of $\mathcal{J}_{\Phi,f,\gamma}$  given by Theorem \ref{existenceintro}. Then, $u_0$ is locally Log-Lipschitz with the following estimate
	$$|u_0(x) - u_0(y)| \le C \ |x-y| \ |\log |x-y||, \quad x,y \in \Omega',$$
	for any $\Omega' \Subset \Omega$, where $C>0$ is a constant which depends on $\|f\|_\infty,\|\gamma\|_1,\Omega',\Omega$ and $N$. In particular, $u_0 \in C_{loc}^{0,\alpha}$, for all $\alpha \in (0,1)$.
\end{theorem}

Furthermore, as is well known, see for example \cite{BM}, the exponential functional
$$I(u) = \int_\Omega \Phi(|\nabla u|)dx,$$
is closely related to the quasilinear operator $\Delta u + 2 \Delta_\infty u$, where
$$ \Delta_\infty u = \sum_{i,j=1}^N \dfrac{\partial u}{\partial x_i}\dfrac{\partial u}{\partial x_j} \dfrac{\partial^2 u}{\partial x_ix_j}.$$
In this way we show that every minimizer of the functional $\mathcal{J}_{\Phi,f,\gamma}$  is a solution, both in the weak sense and in the viscosity sense, of a problem intrinsically related to this operator. More precisely, we have:
\begin{theorem}\label{WV}
	Assume that $u_0 \in W_\psi^{1,\Phi}(\Omega)$ is a minimizer of $\mathcal{J}_{\Phi,f,\gamma}$ . Then, $u_0$ satisfies the equations
	$$-\Delta u - 2 \Delta_\infty u = f_+(x)\exp(-|\nabla u|^2), \; \;  \mbox{in} \; [u > 0],$$
	and
	$$-\Delta u - 2 \Delta_\infty u = f_-(x)\exp(-|\nabla u|^2), \; \;  \mbox{in} \; [u \le 0],$$
	in the weak and viscosity sense.
\end{theorem}

Finally, concerning the free boundaries of the minimizers of $\mathcal{J}_{\Phi,f,\gamma}$, i.e. the sets $\partial \{u_0>0\}$ and $\partial \{u_0<0\}$, we were able to prove that they have a locally finite perimeter. To do so, as better explained in Section \ref{SecFP}, we need to impose an additional condition on the $\gamma$ function, namely, there exists $c_\gamma>0$	
\begin{equation}\label{gc}
	\left\{\begin{aligned}
		&\gamma(x,s)> c_\gamma, \; \mbox{for any} \ s \neq 0;&\\
		&\gamma(x,0)=0,&
	\end{aligned}\right.
\end{equation}
for any $x \in \Omega$. 

\begin{theorem}\label{LPF}
	Let $u_0 \in W_\psi^{1,\Phi}(\Omega)$ be a minimizer of the functional $\mathcal{J}_{\Phi,f,\gamma}$ with $\gamma$ satisfying \eqref{gc}. Then, the reduced free boundaries $\partial^* \{u_0>0\}$ and $\partial^* \{u_0<0\}$ are locally finite perimeter sets. In particular, these reduced free boundaries have at most $N-1$ as Hausdorff dimension.
\end{theorem}

We would like to explain the idea of proving the main results. Due to the fact that $\Phi$ does not satisfy the $\Delta_2$-condition (for more details see Section \ref{SecOS}), the Banach space $W^{1,\Phi}(\Omega)$ lacks reflexivity and separability. As a consequence, it is impossible to employ the minimizing sequences to identify minimizers directly. Furthermore, by direct computation we can see that
$$1\le \dfrac{\varphi'(s)s}{\varphi(s)} \stackrel{s \to +\infty}{\longrightarrow} + \infty,$$
where $\varphi = \Phi'$, which implies that, in this case, the $\Phi$-Laplacian is not an uniformly elliptic operator.
Following \cite{SS} by Santos and Soares, we will introduce for each $k \in \mathbb{N}$ the truncated function $\Phi_k$ defined for $s \in \mathbb{R}$ by
\begin{equation}\label{PT}
	\Phi_k(s) = \sum_{i=1}^{k} \dfrac{1}{i!}|s|^{2i},
\end{equation}
which is an $\mathcal{N}$-function that satisfies $\Delta_2$-condition as well as the condition \eqref{Lie}, for more details see Subsection \ref{sub21}. We will consider the truncated functional
$$\mathcal{J}_{\Phi_k,f,\gamma}(u) = \int_\Omega \Big(\Phi_k(|\nabla u|) - f(x,u)u + \gamma(x,u)\Big)dx$$
defined on 
$$W_\psi^{1,\Phi_k}(\Omega) = \Big\{ u \in W^{1,\Phi_k}(\Omega) \ : \ \ |\nabla u| \in K_{\Phi_k}(\Omega) \ \mbox{and} \ u-\psi \in W_0^{1,\Phi_k} \Big\}.$$ In fact, given $\psi \in W^{1,\Phi}(\Omega)$ with $|\nabla \psi| \in K_\Phi(\Omega)$, we note that by the continuous embeddings $W^{1,\Phi}(\Omega) \hookrightarrow W^{1,\Phi_k}(\Omega)$, for every $k \in \mathbb{N}$, and $W^{1,\Phi_k}(\Omega) \hookrightarrow C^{0,\alpha}(\Omega)$, for some $\alpha\in(0,1)$ and sufficiently large $k$, the function $\psi$ belongs to $W^{1,\Phi_k}(\Omega) \cap C^{0,\alpha}(\Omega)$. Consequently, the truncated problem
\begin{equation}\label{auxiliaryproblem}
	\min \Big\{ \mathcal{J}_{\Phi_k,f,\gamma}(u) \ : \ u \in W^{1,\Phi_k}(\Omega) \ \mbox{and} \  u-\psi \in W_0^{1,\Phi}(\Omega) \Big\},
\end{equation}
is well-defined. Therefore, we will investigate the existence of minimizers $u_k$ for problem \eqref{auxiliaryproblem} first and then analyze the behavior of $u_k$ as $k \to \infty$.

It is worthy noting   that all results proved here for the truncated problem \eqref{auxiliaryproblem} (see Section \ref{SecTP}) as well as the Theorem \ref{LPF} can be generalized, with minor changes, to a general function $G$ satisfying \eqref{Lie}, i.e., we can generalize to any uniformly elliptic operator.


The plan of this paper is as follows: Section \ref{SecOS} is devoted to recalling key properties concerning Orlicz and Orlicz-Sobolev spaces. The subsequent Section \ref{SecTP} is dedicated to establishing the existence, boundedness and regularity for the truncated problem \eqref{auxiliaryproblem}. In Section \ref{SecER} we prove the Theorems \ref{existenceintro}, \ref{reg} and \ref{WV}. Finally, the Section \ref{SecFP} is devoted to proving the Theorem \ref{LPF}.

For simplicity, we will denote $\mathcal{J}_{\Phi,f,\gamma}$ and $\mathcal{J}_{\Phi_k,f,\gamma}$ solely as $\mathcal{J}$ and $\mathcal{J}_k$, respectively. Throughout this paper, we use the following notations: The $N$-dimensional Lebesgue's and Hausdorff's measure of the measurable set $A$ will be denoted by $\mathcal{L}^N(A)$ and $\mathcal{H}^N(A)$, respectively. Let us denote by $[u \ge v]$, respectively $[u \le v]$, the set $[x \in \Omega \ : \ u(x) \ge v(x)]$, respectively $[x \in \Omega \ : \ u(x) \le v(x)]$. $A \Subset B$ means that $A$ is compactly contained in $B$. The usual norms in $L^q(\Omega)$, with $1\le q \le \infty$, will be denoted by $\|\cdot\|_q$. Finally, we denote $\|f\|_\infty = \|f_+\|_\infty + \|f_-\|_\infty$ and $\|\gamma\|_1 = \|\gamma_+\|_1 + \|\gamma_-\|_1$.

\section{Orlicz spaces}\label{SecOS}

In this section, we recall some properties of the Orlicz and Orlicz-Sobolev spaces, which can be found in \cite{A,RaoRen}. First of all, we recall that a continuous function $A : \mathbb{R} \rightarrow [0,+\infty)$ is an $\mathcal{N}$-function if:
\begin{itemize}
	\item[$(i)$] $A$ is convex;
	\item[$(ii)$] $A(t) = 0$ if, and only if, $t = 0$;
	\item[$(iii)$] $\displaystyle\lim_{t\rightarrow0}\frac{A(t)}{t}=0$ and $\displaystyle\lim_{t\rightarrow+\infty}\frac{A(t)}{t}= +\infty$;
	\item[$(iv)$] $A$ is even.
\end{itemize}

The Orlicz space associated to $A$ and an open $U\subset \mathbb{R}^N$ is defined by
$$
L_A(U)= \left\{u \in L^1_{loc}(U) \;: \;
\int_U A \left(\frac{u}{\lambda}\right)\, dx<+\infty, \ \mbox{for some}~ \lambda>0\right\},
$$
and
$$
|u|_{A}=\inf \left\{\lambda>0\;:\;\int_U
A \left(\frac{u}{\lambda} \right)\, dx\leq 1 \right\},
$$
defines a norm (the Luxemburg norm) on $L_A(U)$ and turns this space into a Banach space.

In the study of the Orlicz space $L_A(U)$, we denote by $K_A(U)$ the Orlicz class as the set below
$$K_A(U) \coloneqq \left\{u \in L^1_{loc}(U)\;: \; \int_U A (u)\, dx<+\infty\right\}.$$
In this case, we have
$$K_A(U) \subset L_A(U),$$
where $L_A(U)$ is the smallest subspace containing $K_A(U)$.

The $\mathcal{N}$-function $A$ satisfies the  $\Delta_2$-condition (shortly, $A\in \Delta_2$),
if there exist $k>0$ and $\delta_0\geq0$  such that
$$
A(2t)\leq k A(t),~  t\geq \delta_0.
$$
If $\mathcal{L}^N(U)=\infty$, then we can take $\delta_0=0$. It is possible to prove that if $A \in \Delta_2$ then
$$K_A(U) = L_A(U).$$

A crucial function in examining the characteristics of the Orlicz space $L_A(\Omega)$ is the conjugate function of $A$, denoted as $\widetilde{A}$. This function is an $\mathcal{N}$-function defined by the Legendre transform of $A$, that is
$$\widetilde{A}(t) = \sup \{st - A(s) \ : \ s \in \mathbb{R}\}, \ t \in \mathbb{R}.$$
Using the function $\widetilde{A}$, we have the following inequalities
$$st\leq A(s)+ \widetilde{A}(t),\  s,t\geq0 \,\,\, \mbox{(Young's inequality)}$$
and
$$\int_\Omega uvdx\leq 2\vert u\vert_A\vert v\vert_{\widetilde{A}}, \ u\in L_{A}(\Omega),  v\in
L_{\widetilde{A}}(\Omega) \,\,\, \mbox{(H\"{o}lder's inequality).}$$

\begin{lemma}\label{tilde}
	Let $A$ and $\widetilde{A}$ be a pair of conjugated functions and $a:[0,\infty) \to \mathbb{R}$ be the derivative of $A$. Then the following properties hold:
	\begin{description}
		\item[a)]  $\widetilde{A}(a(t)t) \le A(2t)$, for all $t \ge 0$;
		
		\item[b)] $\widetilde{A}\left(\dfrac{A(t)}{t}\right) \le A(t),$ for all $t >0$.
	\end{description}
\end{lemma}

\begin{definition}
	Let $A$ and $B$ be two $\mathcal{N}$-functions. We say that $B$ dominates $A$ near infinity if there exist positive constants $t_0$ and $k$ such that
	$$A(t) \le B(kt),$$
	for all $t\ge t_0$.
\end{definition}

\begin{lemma}\label{imerAB}
	If $\mathcal{L}^N(U)<\infty$ and $B$ dominates $A$ near infinity, then
	$$L_B(U) \hookrightarrow L_A(U),$$
	continuously.
\end{lemma}

The corresponding Orlicz-Sobolev space is defined by
\[
W^{1, A}(U) = \left\{ u \in L_{A}(U) \ :\ \frac{\partial u}{\partial x_{i}} \in L_{A}(U), \quad i = 1, ..., N\right\},
\]
endowed with the norm
\[
\Vert u \Vert_{1,A} =  \vert \nabla u\vert _{A} +  \vert u\vert_{A},
\]
which turns this space into a Banach space.  Finally, The space $W_0^{1,A}(U)$ is defined as the weak$^*$ closure of $C_0^{\infty}(U)$ in $W^{1,A}(U)$, that is
$$W_0^{1,A}(U) = \overline{C_0^\infty(U)}^{\| \cdot \|_{1,A}}.$$
Furthermore, when $U$ is a smooth bounded domain, a version of the Poincaré's inequality holds, namely
$$|u|_A \le K_U |\nabla u|_A, \quad \mbox{for all} \ u \in W_0^{1,A}(U),$$
where $K_U>0$ depends only on the diameter of $U$.

\subsection{A special case}\label{sub21}

In this subsection we will consider a specific class of $\mathcal{N}$-functions. Let $g: [0,\infty) \to \mathbb{R}$ be a $C^1$-function satisfying
\begin{equation}\label{Ga}
	\delta\le \dfrac{g'(s)s}{g(s)} \le g_0,
\end{equation}
for some $\delta,g_0 \ge 1$. In this scenario, by defining $\displaystyle G(t) = \int_{0}^{t} g(s)ds$, Lieberman \cite{L} proved the following:
\begin{description}
	\item[$(g_1)$] $\min\left\{s^\delta,s^{g_0}\right\}g(t) \le g(st) \le \max\left\{s^\delta,s^{g_0}\right\}g(t)$;
	
	\item[$(g_2)$] $G$ is a convex and $C^2$ function;
	
	\item[$(g_3)$] $\dfrac{sg(s)}{g_0+1} \le G(s) \le sg(s),$ for any $s\ge 0$.
\end{description}
Furthermore, by $(g_1)$ and $(g_3)$ it follows that:
\begin{description}
	\item[$(G_1)$]  $\min\left\{s^{\delta+1},s^{g_0+1}\right\}\dfrac{G(t)}{g_0+1} \le G(st) \le (g_0+1)\max\left\{s^{\delta+1},s^{g_0+1}\right\}G(t)$,
\end{description}
and by convexity of $G$ and this last inequality we get
\begin{description}
	\item[$(G_2)$]  $G(s+t) \le 2^{g_0}(g_0+1)\big(G(s) + G(t)\big)$ for any $s,t>0$.
\end{description}

After that, Fukagai, Ito and Narukawa \cite{FIN} proved that, under condition \eqref{Ga}, $G$ is an $\mathcal{N}$-function that itself and $\widetilde{G}$ satisfy the $\Delta$-condition. Additionally, they showed that:
\begin{description}
	\item[$(G_3)$]  $\displaystyle \min\{|u|_G^{\delta+1},|u|_G^{g_0+1}\} \le \int_U G(|u|)dx \le \max\{|u|_G^{\delta+1},|u|_G^{g_0+1}\}$, for all $u \in L_G(U)$.
\end{description}

Recently, Cantizano, Salort and Spedaletti \cite[see Lemma 3.1]{CSS} proved a refined version of the monotonicity of Young's functions that satisfy \eqref{Ga}, as follows
\begin{equation}\label{monog}
	\left(g(|x|)\dfrac{x}{|x|} - g(|y|)\dfrac{y}{|y|}\right)\cdot(x-y) \ge C(\delta)G(|x-y|), \quad \mbox{for all} \ x,y\in\mathbb{R}^N,
\end{equation}
with $C(\delta)>0$. Furthermore, using the following notation
$$(v)_r = \fint_{B_r} v dx = \dfrac{1}{\mathcal{L}^N(B_r)} \int_{B_r} v dx,$$
since $G$ is convex, the Jensen's inequality implies that
\begin{equation}\label{JI}
	\int_{B_r} G(|(u)_r - (v)_r|) dx \le \int_{B_r}G(|u - v|)dx,
\end{equation}
for any $u,v \in L_{G}(B_r)$. By a similar argument, with aid of $(G_2)$ we get
\begin{equation}\label{Gu(u)r}
	\int_{B_r} G(|u - (u)_r|) dx \le 2^{g_0+1}(g_0+1)\int_{B_r}G(|u - L|)dx,
\end{equation}
for any $L \in \mathbb{R}$ and $u \in L_G(B_r)$.

\begin{lemma}\label{imerAl}
	Assume that $G$ is an $\mathcal{N}$-function satisfying \eqref{Ga}. Then, $L_G(U) \hookrightarrow L^{\delta+1}(U)$ continuously.
\end{lemma}
\begin{proof}
	It is enough to note that, by $(G_1)$, G dominates the function $B(s)=s^{\delta+1}$ near infinity, then the result follows by Lemma \ref{imerAB}.
\end{proof}
A crucial implication of \eqref{Ga} is that
\begin{equation}\label{UE}
	\delta \dfrac{g(|x|)}{|x|} |\xi|^2 \le \dfrac{\partial g_i}{\partial x_j} \xi_i \xi_j \le g_0 \dfrac{g(|x|)}{|x|} |\xi|^2,
\end{equation}
for any $x,\xi \in \mathbb{R}^N$ and $i,j =1,...,N$. This inequality means that the equation
\begin{equation}\label{G-harmonic}
	-\mbox{div} \left(g(|\nabla v|)\dfrac{\nabla v}{|\nabla v|}\right) = 0, \quad \mbox{in} \ U,
\end{equation}
is uniformly elliptic for $\frac{g(|x|)}{|x|}$ bounded and bounded away from zero, for any domain $U\subset \mathbb{R}^N$. Furthermore, if $v \in W^{1,G}(U)$ satisfy \eqref{G-harmonic}, we say that $v$ is $G$-harmonic in $U$.

We notice that the truncated function
$$\Phi_k(s) = \sum_{i=1}^{k} \dfrac{1}{i!}|s|^{2i},$$
satisfies
$$2 \le \dfrac{s\varphi_k'(s)}{\Phi_k(s)} \le 2k, \quad s>0,$$
where $\varphi_k = \Phi_k'$. Consequently, for any $k \in \mathbb{N}$, we can verify that $\Phi_k$ satisfies a stronger version of property $(G_1)$, namely
$$\min\left\{s^2,s^{2k}\right\}\Phi_k(t) \le \Phi_k(st) \le \max\left\{s^{2},s^{2k}\right\}\Phi_k(t), \quad \mbox{for all}\ s,t\ge0.$$
Finally, we have a growth result for the $f$ function with respect of the function $\Phi_k$.

\begin{lemma}\label{growthf}
	Fixed any $k_0 \in \mathbb{N}$, for all $\varepsilon>0$, there exists $M>0$, which depends on $\varepsilon$ and $k_0$, such that
	$$|f(x,t)t| \le \varepsilon \Phi_{k_0}(t), \quad |t| \ge M, \; \; \mbox{a.e.} \ x \in \Omega.$$
	In particular, for any $n \in \mathbb{N}$ and $\varepsilon>0$ there exists $M>0$, which depends on $\varepsilon$ and $n$ such that
	$$\dfrac{|f(x,t)t|}{|t|^{2n}} \le \varepsilon, \quad |t| \ge M, \; \; \mbox{a.e.} \ x \in \Omega.$$
\end{lemma}
\begin{proof}
	Indeed, fixed $x \in \Omega$, it is enough to note that, since $\Phi_{k_0}$ is a polynomial of degree $2k_0$, then
	$$\lim_{t \to +\infty} \dfrac{|f(x,t)t|}{\Phi_{k_0}(t)} =  \lim_{t \to +\infty} \dfrac{|f_+(x)| t}{\sum_{n=1}^{k_0} \frac{1}{n!}t^{2n}} = 0,$$
	and
	$$\lim_{t \to -\infty} \dfrac{|f(x,t)t|}{\Phi_{k_0}(t)} =  \lim_{t \to -\infty} \dfrac{|f_-(x)t|}{\sum_{n=1}^{k_0} \frac{1}{n!}t^{2n}} = 0.$$
	Therefore, given $\varepsilon>0$ there exists $M>0$ such that
	$$\dfrac{|f(x,t)t|}{\Phi_{k_0}(t)} \le \varepsilon, \quad |t|\ge M,$$
	as we desired.	
\end{proof}

\section{The truncated problem}\label{SecTP}

In this section we will prove the results for the truncated functional $\mathcal{J}_k$. Notice that, we have defined the functional $\mathcal{J}_k$ by
$$\mathcal{J}_k(u) = \int_\Omega \Big(\Phi_k(|\nabla u|) - f(x,u)u + \gamma(x,u)\Big)dx,$$
where
$$\Phi_k(s) = \sum_{i=1}^{k} \dfrac{1}{i!}|s|^{2i}.$$

To begin with, let us recall some well-known lemmas that are important in establishing the Log-Lipschitz regularity of the minimizers. Since our approach is to study the truncated problem and then understand its behavior as the parameter $k$ goes to infinity, a better understanding of the constants in the estimates helps to prove the regularity. We will start with a estimate due to Martínez and Wolanski \cite[Theorem 2.3]{MW}. It is noteworthy that, with the help of inequality \eqref{monog}, the proof of this result, unlike the original one presented in \cite{MW}, can be outlined in a few lines, as follows.

\begin{lemma}\label{MartinezWolanski}
	Fixed $k \in \mathbb{N}$, consider $u \in W^{1,\Phi_k}(\Omega)$, $x_0 \in \Omega$ such that $B_r(x_0) \Subset \Omega$, for some $r>0$, and $v$ a $\Phi_k$-harmonic function in $B_r(x_0)$ such that $v-u \in W_0^{1,\Phi_k}(B_r(x_0))$. Then, there exists a constant $C>0$ independent of $k$ such that 
	$$kC\int_{B_r(x_0)} \big(\Phi_k(|\nabla u|) - \Phi_k(|\nabla v|)\big)dx \ge \int_{B_r(x_0)} \Phi_k(|\nabla u - \nabla v|)dx.$$
\end{lemma}
\begin{proof}
	For simplicity, let us denote $B_r(x_0)$ as $B_r$. Setting
	$$u^s \coloneqq su + (1-s)v, \quad \mbox{for any} \ 0\le s \le 1,$$
	by integral form of the mean value theorem, since $\nabla(u-v) = s^{-1}\nabla(u^s-v)$ and $v$ is $\Phi_k$-harmonic, we get
	\begin{eqnarray}\label{CI}
		\int_{B_r} \big(\Phi_k(|\nabla u|) - \Phi_k(|\nabla v|)\big)dx &=& \int_{B_r} \int_{0}^{1} \varphi_k(|\nabla u^s|)\dfrac{\nabla u^s}{|\nabla u^s|} \cdot\nabla (u-v)ds\ dx \nonumber \\
		&=& \int_{B_r} \int_{0}^{1} \dfrac{1}{s}\left(\varphi_k(|\nabla u^s|)\dfrac{\nabla u^s}{|\nabla u^s|} - \varphi_k(|\nabla v|)\dfrac{\nabla v}{|\nabla v|}\right) \cdot\nabla (u^s-v)ds\ dx. \nonumber \\
		&\stackrel{\eqref{monog}}{\ge}& C \int_{B_r} \int_{0}^{1} \dfrac{1}{s} \Phi_k(|\nabla u^s - \nabla v|)dx \nonumber \\
		&\ge& \dfrac{C}{2k} \int_{B_r} \Phi_k(|\nabla u - \nabla v|)dx, \nonumber 
	\end{eqnarray}
	where $C>0$ does not depends on $k$, and this is precisely the assertion of the lemma.
\end{proof}

The proof of the next two lemmas can be obtained by following the same steps as their original proofs with the help of \eqref{Gu(u)r} applied to $\Phi_k$, so their demonstrations will be omitted.

\begin{lemma}[{see \cite[Lemma 5.1]{L}}] \label{Lieberman}
	Let $v \in W^{1,\Phi_k}(B_R(x_0))$ be a $\Phi_k$-harmonic function in $B_R(x_0)$, for some $x_0 \in \mathbb{R}^N$ and $R>0$. Then, for some constant $\sigma \in (0,1)$
	$$\int_{B_r(x_0)} \Phi_k(|\nabla v - (\nabla v)_r|)dx \le 2^{2k+1}k \left(\dfrac{r}{R}\right)^{N+\sigma} \int_{B_R(x_0)} \Phi_k(|\nabla v - (\nabla v)_R|)dx,$$
	holds for any $r \in (0,R)$.
\end{lemma}

\begin{lemma}[{see \cite[Lemma 4.1]{ZZZ}}]\label{Zeng}
	Let $v \in W^{1,\Phi_k}(B_R(x_0))$ be a $\Phi_k$-harmonic function in $B_R(x_0)$, for some $x_0 \in \mathbb{R}^N$ and $R>0$. Then, for any $r \in (0,R)$ and $u \in W^{1,\Phi_k}(B_R(x_0))$
	\begin{eqnarray}
		\int_{B_r(x_0)} \Phi_k(|\nabla u - (\nabla u)_r|)dx &\le&  2^{5(2k+1)}k^5 \left(\dfrac{r}{R}\right)^{N+\sigma}\int_{B_R(x_0)}\Phi_k(|\nabla u - (\nabla u)_R|)dx \nonumber \\
		&& \hspace{1.5cm}+ 2^{4(2k+1)}k^4\int_{B_R(x_0)} \Phi_k(|\nabla u - \nabla v|)dx, \nonumber
	\end{eqnarray}
	holds, where $\sigma \in (0,1)$ is given in Lemma \ref{Lieberman}.
\end{lemma}

Finally, let us recall the following iteration lemma.

\begin{lemma}[{see \cite[Lemma 2.7]{LQO}}]\label{Leitao}
	Let $\phi$ be a non-negative and non-decreasing function. Suppose that
	$$\phi(r) \le A\left(\dfrac{r}{R}\right)^\alpha\phi(R)+ BR^\beta,$$
	for all $0<r\le R \le R_0$, where $A,B,\alpha,\beta$ are positive constants with $\beta<\alpha$. Fixed $\theta \in (0,1)$, choose $\tau \in (\beta,\alpha)$ such that $A\theta^\alpha=\theta^\tau$. Then, for any $0<\sigma\le \beta$ holds
	$$\phi(r) \le c_3 \left(\dfrac{r}{R}\right)^\sigma\big(\phi(R) + BR^\sigma\big),$$
	where $c_3 = \max \left\{\theta^{-\tau}, \dfrac{\theta^{-2\tau}}{1-\theta^{\tau-\beta}}\right\}$.
\end{lemma}

Now we are able to prove the main results of this section. Let us start by proving the existence of minimizers in $W_\psi^{1,\Phi_k}(\Omega)$ of the functional $\mathcal{J}_k$, as well as its uniform boundedness in the $L^\infty$-norm.

\begin{theorem}\label{ukLestimate}
	For any $k \in \mathbb{N}$, there exists a minimizer $u_k \in W_\psi^{1,\Phi_k}(\Omega)$ to $\mathcal{J}_k$. Furthermore, $u_k \in L^\infty(\Omega)$ and there exists a positive constant $C$ which depends only on $\displaystyle \|f\|_\infty,\mathcal{L}^N(\Omega),N,\sup_{\overline{\Omega}}\psi$, but does not depend on $k$, such that
	$$\|u_k\|_\infty \le C.$$
\end{theorem}
\begin{proof}
	First of all, since the space $W_\psi^{1,\Phi_k}(\Omega)$ is now reflexive and separable we can adopt the usual approach to prove the existence of minimizing, showing that the functional is bounded from below and thus considering a minimizing sequence. Now, to simplify our notation, let us refer to the minimizer of $\mathcal{J}_k$ as $u$. Our initial focus is on to establish the upper boundedness of $u$. The idea is use \cite[see Chapter 2, Lemma 5.2]{LU}. To begin, consider $j_0$ the smallest integer greater than $\sup_{\overline{\Omega}} \psi$, i.e.
	$$j_0 = \min \left\{ j \in \mathbb{N} \; \; : \; \; j \ge \sup_{\overline{\Omega}} \psi\right\}.$$
	For each $j \ge j_0$ define $u_j : \Omega \to \mathbb{R}$ by
	$$u_j \coloneqq\left \{
	\begin{array}{rclcl}
		j \cdot \mbox{sign}(u), &\mbox{if}& |u|>j;\\
		u, & \mbox{if}  & |u|\le j, \\
	\end{array}
	\right.$$
	where $\mbox{sign}(u)=1$ if $u>0$ and $\mbox{sign}(u)=-1$ if $u\le0$. Also we define the set $A_j \coloneqq [|u|>j]$ and we can observe that
	$$u_j = u \; \; \mbox{in} \ A_j^c \quad \mbox{and} \quad u_j = j \cdot \mbox{sign}(u) \; \; \mbox{in} \ A_j,$$
	as well as
	$$\nabla u_j = \nabla u \; \; \mbox{a.e. in} \ A_j^c \quad \mbox{and} \quad \nabla u_j = 0 \; \; \mbox{a.e. in} \ A_j.$$
	Since $u$ and $u_j$ have the same sign, by minimality of $u$ we get
	\begin{eqnarray}
		\int_{A_j} \Phi_k(|\nabla u|)dx &=& \int_\Omega \big(\Phi_k(|\nabla u|)- \Phi_k(|\nabla u_j|)\big)dx \nonumber \\
		&\le& - \int_\Omega \big(f(x,u_j)u_j - f(x,u)u\big)dx - \int_\Omega \big(\gamma(x,u_j) - \gamma(x,u)\big)dx \nonumber \\
		&=& - \int_{A_j} f(x,u)(u_j-u)dx \nonumber \\
		&\le& 2 \int_{A_j} |f(x,u)| \ (|u|-j)dx.\nonumber
	\end{eqnarray}
	It is noteworthy that, given the manner in which $j$ was considered, we have $(|u| - j)^+ \in W_0^{1,2}(\Omega)$ and the function $(|u| - j)^+$ within $\Omega$ aligns with $(|u| - j)$ in $A_j$. Thus, applying H\"{o}lder's and Poincaré's inequalities, we deduce:
	\begin{eqnarray}
		\int_{A_j} \Phi_k(|\nabla u|)dx &\le& 2 \|f\|_\infty C_\Omega \left(\int_{A_j} |\nabla (|u|-j)|^2dx\right)^{\frac{1}{2}} \mathcal{L}^N(A_j)^{\frac{1}{2}} \nonumber \\
		&=& 2 \|f\|_\infty C_\Omega \left(\int_{A_j} |\nabla u|^2dx\right)^{\frac{1}{2}} \mathcal{L}^N(A_j)^{\frac{1}{2}}, \quad j \ge j_0. \nonumber
	\end{eqnarray}
	Now, by Young's inequality, namely
	\begin{equation}\label{YI}
		ab \le \dfrac{q-1}{q}a^{\frac{q}{q-1}} + \dfrac{1}{q} b^q, \quad a,b \in \mathbb{R}_+, \; \; q>1,
	\end{equation}
	with
	$$a = \varepsilon^{1/2} \left( \int_{A_j} |\nabla u|^2\right)^{\frac{1}{2}}, \qquad b= 2\|f\|_\infty C_\Omega \mathcal{L}^N(A_j)^{\frac{1}{2}} \varepsilon^{-1/2},$$
	$q=2$ and $\varepsilon>0$ to be defined, we obtain
	$$\int_{A_j} \Phi_k(|\nabla u|)dx \le \dfrac{\varepsilon}{2} \int_{A_j}|\nabla u|^2dx + 2\|f\|_\infty^2 \ C_\Omega^2 \ \mathcal{L}^N(A_j)\varepsilon^{-1},$$
	and then by definition of $\Phi_k$ we have
	\begin{eqnarray}
		\left(\dfrac{1}{2}-\dfrac{\varepsilon}{2}\right)\int_{A_j} |\nabla u|^2dx &\le& \int_{A_j} \Phi_k(|\nabla u|)dx - \dfrac{\varepsilon}{2}\int_{A_j} |\nabla u|^2 dx \nonumber \\
		&\le& 2\|f\|_\infty^2 \ C_\Omega^2 \ \mathcal{L}^N(A_j)\varepsilon^{-1}. \nonumber
	\end{eqnarray}
	Choosing $\varepsilon = \frac{1}{2}$ we get
	$$\int_{A_j} |\nabla u|^2 dx \le 4\|f\|_\infty^2 \ C_\Omega^2 \ \mathcal{L}^N(A_j),$$
	for any $j \ge j_0$. Finally, considering $\beta \le 2/N$, since $1-\frac{2}{N}+\beta\le 1$ we have
	\begin{eqnarray}\label{graduAj}
		\int_{A_j} |\nabla u|^2 dx &\le& 4\|f\|_\infty^2 \ C_\Omega^2 \ \mathcal{L}^N(\Omega) \cdot \left(\dfrac{\mathcal{L}^N(A_j)}{\mathcal{L}^N(\Omega)}\right)^{1-\frac{2}{N}+\beta} \nonumber \\
		&=& 4\|f\|_\infty^2 \ C_\Omega^2 \ \mathcal{L}^N(\Omega)^{\frac{2}{N}-\beta} \ \mathcal{L}^N(A_j)^{1-\frac{2}{N}+\beta} \nonumber \\
		&\eqqcolon& C\big(\|f\|_\infty, N, \mathcal{L}^N(\Omega)\big) \ \mathcal{L}^N(A_j)^{1-\frac{2}{N}+\beta}, \quad j \ge j_0.
	\end{eqnarray}
	In this case, we can apply the Lemma 5.2 of \cite[see Chapter 2]{LU} with $m=2$, $\alpha=0$ and $\gamma= C\big(\|f\|_\infty, N, \mathcal{L}^N(\Omega)\big)$ which implies that
	\begin{eqnarray}\label{supu}
		\sup_\Omega u &\le& j_0 +  C\big(\|f\|_\infty, N, \mathcal{L}^N(\Omega)\big) \left(\int_{A_{j_0}} (u-j_0)dx\right)^{\frac{\beta}{2+\beta}}.
	\end{eqnarray}
	We claim that
	\begin{equation}\label{normone}
		\int_{A_{j_0}} (u-j_0)dx \le C,
	\end{equation}
	where $C>0$ is a constant independent of $k$. Indeed, it enough to observe that by H\"{o}lder's and Poincaré's inequalities and \eqref{graduAj} we get
	\begin{eqnarray}
		\int_{A_{j_0}} (u-j_0)dx &\le& \mathcal{L}^N(\Omega)^{\frac{1}{2}} \ C_\Omega \left(\int_{A_{j_0}} |\nabla u|^2dx\right)^{\frac{1}{2}} \nonumber \\
		&\le& C\big(\|f\|_\infty, N, \mathcal{L}^N(\Omega)\big) \ \mathcal{L}^N(A_{j_0})^{1-\frac{1}{N}+\frac{\beta}{2}} \nonumber \\
		&\le&  C\big(\|f\|_\infty, N, \mathcal{L}^N(\Omega)\big) \ \mathcal{L}^N(\Omega)^{1-\frac{1}{N}+\frac{\beta}{2}}, \nonumber
	\end{eqnarray}
	as claimed. Therefore, the upper boundedness of $u$ follows by combining \eqref{supu} and \eqref{normone}. By similar procedure we prove that $\inf_\Omega u \ge -C$, where $C>0$ is independent of $k$. This proves the proposition.
\end{proof}

As widely acknowledged in the literature, the $L^\infty$ estimate for the minimizer of $\mathcal{J}_k$ provides us with a universally applicable control over $u_k$ in $W_\psi^{1,\Phi_k}(\Omega)$. Indeed, note that
\begin{eqnarray}\label{uniestuk}
	\int_\Omega \Phi_k(|\nabla u_k|)dx &\le& \mathcal{J}_k(\psi) + \int_\Omega f(x,u_k)u_kdx - \int_\Omega \gamma(x,u_k)dx \nonumber \\
	&\le& \mathcal{J}(\psi) + \|u_k\|_\infty \|f\|_\infty + \|\gamma\|_1 \nonumber \\
	&\le& C_0,
\end{eqnarray}
where $C_0=C_0\big(\|f\|_\infty,\|\gamma\|_1,\psi, N, \mathcal{L}^N(\Omega)\big)>0$.

Having established the existence and boundedness of minimizers, our attention now turns to investigating their regularity.

\begin{theorem}\label{regk}
	Let $u_k \in W_\psi^{1,\Phi_k}(\Omega)$ be a minimizer of $\mathcal{J}_k$. Then, $\nabla u_k \in BMO_{loc}(\Omega)$. More precisely, for any $\Omega' \Subset \Omega$, there exists a constant $\widetilde{C}=\widetilde{C}(\|f\|_\infty,\|\gamma\|_1,N,\Omega',\Omega)>0$ such that
	\begin{equation}\label{gradientBMOwnk}
		\|\nabla u_k \|_{BMO_{loc}(\Omega)} \le \widetilde{C} \Big( \mbox{dist}(\Omega', \partial\Omega)^{-N} + \mbox{diam}(\Omega)^N\Big)^{\frac{1}{2}}.
	\end{equation}
	Furthermore, $u_k$ is locally Log-Lipschitz with the following estimate
	\begin{equation}\label{uxy}
		|u_k(x) - u_k(y)| \le C(N) \|\nabla u_k \|_{BMO_{loc}(\Omega)} \ |x-y| \ |\log |x-y||.
	\end{equation}
	In particular, $u_k \in C_{loc}^{0,\alpha}(\Omega)$, for all $\alpha \in (0,1)$.
\end{theorem}
\begin{proof}
	Let $y \in \Omega'$ and $R\coloneqq\mbox{dist}(\Omega',\partial\Omega)$. To simplify our notation, we will consider $u_k=u$ and $y=0$. The idea is to use Lemma \ref{Zeng} to estimate $\|\nabla u_k \|_{BMO_{loc}(\Omega)}$. To do this, consider also $v:B_R \to \mathbb{R}$ a $\Phi_k$-harmonic function such that $v-u \in W_0^{1,\Phi_k}(B_R)$. Let us start estimating $\int_{B_R} \Phi_k(|\nabla u - \nabla v|)dx$. By Lemma \ref{MartinezWolanski}
	\begin{eqnarray}\label{pkuv}
		\int_{B_R} \Phi_k(|\nabla u - \nabla v|)dx &\le& k \int_{B_R} \Big(\Phi_k(|\nabla u|) - \Phi_k(|\nabla v|)\Big)dx \nonumber \\
		&\le& k \left(\int_{B_R}\big(f(x,u)u - f(x,v)v\big)dx + \int_{B_R}\big(\gamma(x,v) - \gamma(x,u)\big)dx\right) \nonumber \\
		&\le&  k \left(\int_{B_R}\big(f(x,u)u - f(x,v)v\big)dx + \sup_{\overline{\Omega}} |\gamma| \ C(N)R^N\right),
	\end{eqnarray}
	provide that $u$ is a minimizer of $\mathcal{J}_k$. Furthermore, since 
	$$\int_{B_R}\big(f(x,u)u - f(x,v)v\big)dx \le \|f\|_\infty\int_{B_R} |u-v|dx,$$
	the H\"{o}lder and Poincaré's inequalities imply that
	\begin{eqnarray}
		\int_{B_R}\big(f(x,u)u - f(x,v)v\big)dx &\le& \|f\|_\infty \mathcal{L}^N(B_R)^{1/2} \left(\int_{B_R} |u-v|^2dx\right)^{1/2} \nonumber \\
		&\le& \|f\|_\infty C(N)R^{N/2} \left(\int_{B_R} |\nabla u-\nabla v|^2dx\right)^{1/2}. \nonumber
	\end{eqnarray}
	Now, by Young's inequality (see \eqref{YI}) with $q=2$,
	$$a = k^{-2}\left(\int_{B_R} |\nabla u-\nabla v|^2dx\right)^{1/2} \quad \text{and} \quad b=\|f\|_\infty C(N)R^{N/2} k^2,$$
	we get
	\begin{equation}\label{f1}
		\int_{B_R}\big(f(x,u)u - f(x,v)v\big)dx \le \dfrac{1}{2k} \int_{B_R}|\nabla u - \nabla v|^2dx + \dfrac{k\|f\|_\infty^2C(N)R^N}{2}.
	\end{equation}
	In addition, by definition of $\Phi_k$ we know that
	$$\int_{B_R}|\nabla u - \nabla v|^2dx \le \int_{B_R} \Phi_k(|\nabla u - \nabla v|)dx,$$
	in which it follows by \eqref{f1} that
	\begin{eqnarray}\label{fuv}
		\int_{B_R}\big(f(x,u)u - f(x,v)v\big)dx &\le& \dfrac{1}{2k} \int_{B_R}\Phi_k(|\nabla u - \nabla v|)dx + \dfrac{k\|f\|_\infty^2C(N)R^N}{2}.
	\end{eqnarray}
	By combining \eqref{pkuv} and \eqref{fuv} we have
	\begin{eqnarray}
		\int_{B_R}\Phi_k(|\nabla u - \nabla v|)dx &\le& \dfrac{1}{2}\int_{B_R} \Phi_k(|\nabla u - \nabla v|)dx + \dfrac{k^2\|f\|_\infty^2C(N)R^N}{2} + k\sup_{\overline{\Omega}} |\gamma| C(N)R^N \nonumber \\
		&\le&\dfrac{1}{2}\int_{B_R} \Phi_k(|\nabla u - \nabla v|)dx +  C(N,f,\gamma)k^2 R^N, \nonumber
	\end{eqnarray}
	or even,
	\begin{equation}\label{Pkuv}
		\int_{B_R}\Phi_k(|\nabla u - \nabla v|)dx \le C(N,f,\gamma)k^2 R^N.
	\end{equation}
	By Lemma \ref{Zeng}, for any $r \in (0,R)$, we get
	\begin{eqnarray}
		\int_{B_r}\Phi_k(|\nabla u - (\nabla u)_r|)dx &\le& 2^{5(2k+1)}k^5 \left(\dfrac{r}{R}\right)^{N+\sigma}\int_{B_R}\Phi_k(|\nabla u - (\nabla u)_R|)dx \nonumber \\
		&& \hspace{1.5cm}+ 2^{4(2k+1)}k^4\int_{B_R} \Phi_k(|\nabla u - \nabla v|)dx, \nonumber \\
		&\stackrel{\eqref{Pkuv}}{\le}& 2^{5(2k+1)}k^5 \left(\dfrac{r}{R}\right)^{N+\sigma}\int_{B_R}\Phi_k(|\nabla u - (\nabla u)_R|)dx \nonumber \\
		&& + 2^{4(2k+1)}k^6C(N,f,\gamma)R^N, \nonumber
	\end{eqnarray}
	Furthermore, we can use the Lemma \ref{Leitao} with $\alpha=N+\sigma$, $\beta=N$ and, for $k$ sufficiently large, $c_3 = (2^{5(2k+1)}k^5\theta)^{-1}$, for some $\theta \in (0,1)$ independent of $k$, which implies that
	
	\begin{eqnarray}\label{part1}
		\int_{B_r}\Phi_k(|\nabla u - (\nabla u)_r|)dx &\le& \dfrac{1}{2^{5(2k+1) }k^5\theta} \left(R^{-N} \int_{B_R}\Phi_k(|\nabla u - (\nabla u)_R|)dx  \right. \nonumber \\
		&& \hspace{4cm} + 2^{4(2k+1)}k^6C(N,f,\gamma)R^N\Big)r^N \nonumber \\
		&\stackrel{\eqref{Gu(u)r}}{\le}&\dfrac{1}{2^{5(2k+1)}k^5\theta}\left(2^{2k+1}k\ R^{-N}\int_{B_R}\Phi_k(|\nabla u|)dx \right. \nonumber \\
		&& \hspace{4cm} + 2^{4(2k+1)}k^6C(N,f,\gamma)R^N\Big)r^N \nonumber \\
		&\stackrel{\eqref{uniestuk}}{\le}& \dfrac{1}{2^{5(2k+1)}k^5\theta} \Big(2^{2k+1}k\ R^{-N}C_0 + 2^{4(2k+1)}k^6C(N,f,\gamma)R^N\Big)r^N \nonumber \\
		&\le& \dfrac{k}{2^{2k+1}}\Big(C_0\mbox{dist}(\Omega',\partial\Omega)^{-N} + C(N,f,\gamma)\mbox{diam}(\Omega)^N\Big)r^N, \nonumber \\
	\end{eqnarray}
	where $C_0>0$, given in \eqref{uniestuk}, does not depends on $k$. On the other hand, by definition of $\Phi_k$ and Jensen's inequality we get
	\begin{equation}\label{part2}
		\int_{B_r}\Phi_k(|\nabla u - (\nabla u)_r|)dx \ge \left(\int_{B_r} |\nabla u - (\nabla u)_r|dx\right)^{2}.
	\end{equation}
	By combining \eqref{part1} and \eqref{part2} we obtain
	$$\int_{B_r} |\nabla u - (\nabla u)_r|dx \le\left(\dfrac{k}{2^{2k+1}}\right)^{\frac{1}{2}}\Big(C_0\mbox{dist}(\Omega',\partial\Omega)^{-N} + C(N,f,\gamma)\mbox{diam}(\Omega)^N\Big)^{\frac{1}{2}} r^{\frac{N}{2}},$$
	and then, for $k$ large enough we have
	$$\int_{B_r} |\nabla u - (\nabla u)_r|dx \le \widetilde{C}\Big(\mbox{dist}(\Omega',\partial\Omega)^{-N} + \mbox{diam}(\Omega)^N\Big)^{\frac{1}{2}}r^{\frac{N}{2}},$$
	where $\widetilde{C}>0$ depends on  $\|f\|_\infty,\|\gamma\|_1,\psi,N,\mathcal{L}^N(\Omega)$ and this complete the proof of \eqref{gradientBMOwnk}. The second part follows immediately by \cite[Lemma 1]{AB}.
\end{proof}

\section{Existence and regularity for minimizers of $\mathcal{J}$}\label{SecER}

In this section, we focus on the original functional $\mathcal{J}$. Here we will prove Theorems \ref{existenceintro}, \ref{reg} and \ref{WV}, stated in the introduction.

\subsection{Existence and regularity}

In this subsection, we will prove the Theorems \ref{existenceintro} and \ref{reg}. Let us start with the existence and boundedness.

\begin{proof}[{Proof of Theorem \ref{existenceintro}}]
	Let $u_k$ a minimizer of $\mathcal{J}_k$ in $W_\psi^{1,\Phi_k}(\Omega)$ and any $v \in W_\psi^{1,\Phi}(\Omega)$. Then, we get
	\begin{equation}\label{9}
		\int_\Omega \Big(\Phi_k(|\nabla u_k|) - f(x,u_k)u_k + \gamma(x,u_k)\Big) dx
		\le \int_\Omega \Big(\Phi(|\nabla v|) - f(x,v)v + \gamma(x,v)\Big) dx.
	\end{equation}
	Fixed $n \in \mathbb{N}$, let us start proving that $\{u_k\}_{k\ge n}$ is bounded in $W^{1,2n}(\Omega)$. To do this, choose
	$$0<\varepsilon \le \dfrac{1}{n!2^{4n}C_\Omega},$$
	where $C_\Omega>0$ is the Poincaré constant for the space $W_0^{1,2n}(\Omega)$. By choice of $\varepsilon$, we get
	$$C(n,\Omega) \coloneqq \dfrac{1}{n!} - 2^{4n}C_\Omega\varepsilon \ge 0.$$
	By Lemma \ref{growthf}, we have
	\begin{eqnarray}\label{10}
		\int_\Omega f(x,u_k)u_k dx &\le& C(\|f\|_\infty, \mathcal{L}^N(\Omega)) + 2^{2n}\varepsilon \int_\Omega |u_k - \psi|^{2n} dx + 2^{2n}\varepsilon ||\psi||_{2n}^{2n} \nonumber \\
		&\le& 2^{4n}C_\Omega\varepsilon \int_\Omega |\nabla u_k|^{2n} dx + C\big(\|f\|_\infty,n,C_\Omega,\varepsilon,||\psi||_{1,2n},\mathcal{L}^N(\Omega)\big).
	\end{eqnarray}
	Thus, by combining \eqref{9}, \eqref{10} and using the definition of $\Phi_k$ we get that
	\begin{eqnarray}\label{*}
		C(n,\Omega)\int_\Omega |\nabla u_k|^{2n} dx &\le& \int_\Omega \Phi_k(|\nabla u_k|)dx -2^{4n}C_\Omega\varepsilon \int_\Omega |\nabla u_k|^{2n} dx \nonumber \\
		&\le& \int_\Omega \Phi_k(|\nabla u_k|)dx - \int_\Omega f(x,u_k)u_kdx + \int_\Omega \gamma(x,u_k)dx - \int_\Omega \gamma(x,u_k)dx \nonumber \\
		&& + \  C\big(\|f\|_\infty,\|\gamma\|_1,n,C_\Omega,\varepsilon,||\psi||_{1,2n},\mathcal{L}^N(\Omega)\big), \nonumber\\
		&\le& \int_\Omega \Big(\Phi(|\nabla v|) - f(x,v)v + \gamma(x,v)\Big) dx \nonumber \\
		&& + \  C\big(\|f\|_\infty,\|\gamma\|_1,n,C_\Omega,\varepsilon,||\psi||_{1,2n},\mathcal{L}^N(\Omega)\big),
	\end{eqnarray}
	for all $k \ge n$. Now, define $v_k=u_k - v$ and note that $v_k \in W_0^{1,2n}(\Omega)$, for all $k \ge n$. Indeed, it is enough observe that
	$$\int_\Omega |u_k - v|^{2n}dx \le n! \int_\Omega \Phi_k(u_k - v)dx < \infty,$$
	and the same for the gradient. Then, by Poincaré's inequality for the space $W_0^{1,2n}(\Omega)$
	\begin{equation}\label{PIW2n}
		\left(\int_\Omega |v_k|^{2n}dx\right)^{\frac{1}{2n}} \le C_\Omega \left(\int_\Omega |\nabla v_k|^{2n}dx\right)^{\frac{1}{2n}},
	\end{equation}
	for all $k\ge n$, where $C_\Omega>0$ is a constant independent of $k$. Consequently, by \eqref{PIW2n} we get
	\begin{eqnarray}
		\left(\int_\Omega |u_k|^{2n}dx\right)^{\frac{1}{2n}} &\le& \left(\int_\Omega |u_k-v|^{2n}dx\right)^{\frac{1}{2n}} + \left(\int_\Omega |v|^{2n}dx\right)^{\frac{1}{2n}} \nonumber \\
		&\le& C_\Omega \left(\int_\Omega |\nabla v_k|^{2n}dx\right)^{\frac{1}{2n}} + \left(\int_\Omega |v|^{2n}dx\right)^{\frac{1}{2n}}. \nonumber
	\end{eqnarray}
	Now, by definition of $v_k$ and triangle inequality we obtain
	\begin{eqnarray}\label{**}
		\left(\int_\Omega |u_k|^{2n}dx\right)^{\frac{1}{2n}}  &\stackrel{\eqref{*}}{\le}& C_\Omega \left(\int_\Omega |\nabla v|^{2n}dx\right)^{\frac{1}{2n}} + \left(\int_\Omega |v|^{2n}dx\right)^{\frac{1}{2n}} \nonumber \\
		&& + \ C_\Omega \Bigg(\dfrac{1}{C(n,\Omega)} \int_\Omega \Big(\Phi(|\nabla v|) - f(x,v)v + \gamma(x,v)\Big) dx \nonumber \\
		&& \hspace{2.5cm}+ \ C\big(\|f\|_\infty,n,C_\Omega,\varepsilon,||\psi||_{1,2n},|\Omega|,C(\gamma)\big)\Bigg)^{\frac{1}{2n}},
	\end{eqnarray}
	for any $k \ge n$. As the constants are independent of $k$, it can be deduced from \eqref{*} and \eqref{**} that the sequence $\{u_k\}_{k \ge n}$ is bounded in the Sobolev space $W^{1,2n}(\Omega)$ with respect to $k$, as desired. Employing a diagonal argument, for each $n \in \mathbb{N}$, there exists a subsequence, still denoted by $\{u_k\}_{k \ge n}$, and a function $u_0 \in W^{1,2n}(\Omega)$ such that
	$$u_k \rightharpoonup u_0 \ \mbox{in} \ W^{1,2n}(\Omega) \ \mbox{as} \ k \to \infty,$$
	where $u_0$ is independent of $n$. Moreover, the compact embeddings of $W^{1,2n}(\Omega) \hookrightarrow L^{2n}(\Omega)$ and $W^{1,2n}(\Omega) \hookrightarrow C(\overline{\Omega})$ imply the following convergences:
	$$\left\{\begin{aligned}
		&u_k \to u_0, \; \; \text{in} \; \; L^{2n}(\Omega);&\\
		&u_0 = \psi, \; \; \text{on} \; \; \partial\Omega;&\\
		&u_k(x) \to u_0(x), \; \; \text{a.e. in} \ \Omega.&
	\end{aligned}\right.$$
	Combining these convergences with the fact that the sequence $\{\nabla u_k\}_{k \ge n}$ is bounded in $\big(L^{2n}(\Omega)\big)^N$, we can infer that
	$$\nabla u_k \rightharpoonup \nabla u_0 \;\; \mbox{in} \;\; \big(L^{2n}(\Omega)\big)^N.$$
	Furthermore, by Lebesgue's Theorem with aid of Lemma \ref{growthf} we have
	$$\lim_{k \to \infty} \int_\Omega f(x,u_k)u_kdx = \int_\Omega f(x,u_0)u_0 dx \quad \mbox{and} \quad \lim_{k \to \infty} \int_\Omega \gamma(x,u_k)dx = \int_\Omega \gamma(x,u_0)dx.$$
	As a consequence, for any $m \in \mathbb{N}$, it follows from the sequentially weak lower semi-continuity of the norm that
	\begin{eqnarray}
		&&\sum_{n=1}^m \dfrac{1}{n!} \int_\Omega |\nabla u_0|^{2n} dx - \int_\Omega f(x,u_0)u_0dx + \int_\Omega \gamma(x,u_0)dx \nonumber \\
		&\le& \sum_{n=1}^m \dfrac{1}{n!}\liminf_{k \to \infty} \int_\Omega |\nabla u_k|^{2n} dx - \lim_{k \to \infty} \int_\Omega f(x,u_k)u_kdx + \lim_{k \to \infty} \int_\Omega\gamma(x,u_k)dx \nonumber \\
		&=& \liminf_{k \to \infty} \left(\sum_{n=1}^m \dfrac{1}{n!} \int_\Omega |\nabla u_k|^{2n} dx - \int_\Omega f(x,u_k)u_kdx + \int_\Omega\gamma(x,u_k)dx \nonumber \right),
	\end{eqnarray}
	and so by \eqref{9} we get
	\begin{eqnarray}\label{d}
		&&\sum_{n=1}^m \dfrac{1}{n!} \int_\Omega |\nabla u_0|^{2n} dx - \int_\Omega f(x,u_0)u_0dx + \int_\Omega \gamma(x,u_0)dx \nonumber \\
		&\le& \liminf_{k \to \infty} \left(\sum_{n=1}^k \dfrac{1}{n!} \int_\Omega |\nabla u_k|^{2n} dx - \int_\Omega f(x,u_k)u_kdx + \int_\Omega\gamma(x,u_k)dx \nonumber \right) \nonumber \\
		&\le& \liminf_{k \to \infty} \left(\sum_{n=1}^k \dfrac{1}{n!} \int_\Omega |\nabla v|^{2n} dx - \int_\Omega f(x,v)vdx + \int_\Omega\gamma(x,v)dx \right) \nonumber \\
		&=& \int_\Omega \Big( \Phi(|\nabla v|) - f(x,v)v + \gamma(x,v)\Big) dx,
	\end{eqnarray}
	for any $v \in W_\psi^{1,\Phi}(\Omega)$. Finally, taking the limit as $m \to \infty$ in \eqref{d}, we conclude that
	$$\int_\Omega \Big( \Phi(|\nabla u_0|) - f(x,u_0)u_0 + \gamma(x,u_0)\Big) dx \\ \le \int_\Omega \Big( \Phi(|\nabla v|) - f(x,v)v + \gamma(x,v)\Big) dx,$$
	which shows, by arbitrary of $v \in W_\psi^{1,\Phi}(\Omega)$, that $u_0$ is a minimizer of $\mathcal{J}$ in $W_\psi^{1,\Phi}(\Omega)$, as desired.	
	
	Finally, since $u_0$ is the uniform limit of the sequence $\{u_k\}_{k \in \mathbb{N}}$ of minimizers of $\mathcal{J}_k$, we get by Theorem \ref{ukLestimate} that there exists a constant $C>0$ such that
	$$\|u_0\|_\infty \le C,$$
	where $C$ depends on $\displaystyle \|f\|_\infty,\mathcal{L}^N(\Omega),N,\sup_{\overline{\Omega}}\psi$, and this is precisely the assertion of the theorem.
\end{proof}

\begin{remark}
	Similarly to the previous discussion (see \eqref{uniestuk}), an immediate consequence of the $L^\infty$ estimate in the previous theorem is a uniform control of the $W^{1,\Phi}(\Omega)$-norm of the minimizers of functional $\mathcal{J}$. This control depends solely on universal constants and intrinsic parameters, as elaborated in \eqref{uniestuk}.
\end{remark}

Now, with the aid of Theorems \ref{ukLestimate} and \ref{regk} we will prove the Log-Lipschitz regularity for this minimizers.

\begin{proof}[Proof of Theorem \ref{reg}]
	It is sufficient to note that, since $u_0$ is uniform limit of the sequence $\{u_k\}$ we get
	\begin{eqnarray}
		|u_0(x) - u_0(y)| &\le& |u_0(x) -u_k(x)| + |u_0(y) -u_k(y)| + |u_k(x) -u_k(y)| \nonumber \\
		&\stackrel{\eqref{gradientBMOwnk}}{\le}& \dfrac{1}{2} + \dfrac{1}{2} + C \|\nabla u_k \|_{BMO_{loc}(\Omega)} \ |x-y| \ |\log |x-y|| \nonumber \\
		&\stackrel{\eqref{uxy}}{\le}& C|x-y| \ |\log |x-y||, \nonumber
	\end{eqnarray}
	where $C>0$ does not depends on $k$.
\end{proof}

\subsection{The operator $\Delta u + 2 \Delta_\infty u$}

In this subsection we will prove that the function $u_0$ given in Theorem \ref{existenceintro} is also a solution of a class of problems involving the operator
$$L(u) \coloneqq \Delta u + 2 \Delta_\infty u,$$
where
$$ \Delta_\infty u = \sum_{i,j=1}^N \dfrac{\partial u}{\partial x_i}\dfrac{\partial u}{\partial x_j} \dfrac{\partial^2 u}{\partial x_ix_j}.$$
$L$ is recognized as the Euler-Lagrange operator corresponding to the functional associated with the integral
$$\int_\Omega \Phi(|\nabla u|)dx.$$
In fact, the next lemma states precisely this relation between the exponential functional and the operator $L$. This establishes the basis for our proof. 

\begin{lemma}\label{inftyL}
	Let $U \in \mathbb{R}^N$ be a bounded domain, $u \in W^{1,\Phi}(U)$ and $h \in L^\infty(U)$. Then, $u$ satisfies
	$$-\mbox{div} \left(\varphi(|\nabla u|)\dfrac{\nabla u}{|\nabla u|}\right) = h(x), \ \mbox{in} \ U,$$
	if, and only if,
	$$-\Delta u - 2 \Delta_\infty u = h(x)\exp(-|\nabla u|^2) \ \mbox{in} \ U,$$
	both in distributional sense.
\end{lemma}

Now that we have established the relationship between operators, let us return to the initial objective. The next lemma is pivotal in achieving this goal.

\begin{lemma}\label{subsol}
	Let $u_0 \in W_\psi^{1,\Phi}(\Omega)$ be a minimizer of $\mathcal{J}$, given in Theorem \ref{existenceintro}. Then, $u_0$ is subsolution, i.e.,
	$$0 \ge \int_\Omega \varphi(|\nabla u_0|) \dfrac{\nabla u_0}{|\nabla u_0|}\nabla v \ dx - \int_\Omega f(x,u_0)v\ dx,$$
	for any $v \in C_0^\infty(\Omega)$, with $v\ge 0$.
\end{lemma}
\begin{proof}
	For $v \in C_0^\infty(\Omega)$, with $v \ge 0$, and $\varepsilon>0$, by minimality of $u_0$, we get
	\begin{eqnarray}\label{Pfg}
		0 &\le& \dfrac{1}{\varepsilon}\big(\mathcal{J}(u_0-\varepsilon v) - \mathcal{J}(u_0)\big) \nonumber \\
		&=&\dfrac{1}{\varepsilon} \int_\Omega \Big( \Phi(|\nabla u_0 - \varepsilon\nabla v|) - \Phi(|\nabla u_0|)\Big)dx - \dfrac{1}{\varepsilon} \int_\Omega\big(f(x,u_0 -\varepsilon v)(u_0-\varepsilon v) - f(x,u_0)u_0\big) \nonumber \\
		&& + \dfrac{1}{\varepsilon} \int_\Omega\big(\gamma(x,u_0 -\varepsilon v)- \gamma(x,u_0)\big)dx.
	\end{eqnarray}	
	We claim that
	\begin{equation}\label{64}
		\lim_{\varepsilon \to 0} \dfrac{1}{\varepsilon} \int_\Omega\big(f(x,u_0 -\varepsilon v)(u_0-\varepsilon v) - f(x,u_0)u_0\big) = - \int_\Omega f(x,u_0)vdx.
	\end{equation}
	
	Indeed, note that
	\begin{eqnarray}\label{qf+-}
		\dfrac{f(x,u_0 -\varepsilon v)(u_0-\varepsilon v) - f(x,u_0)u_0}{\varepsilon} &=& \dfrac{f_+(x)}{\varepsilon} \big(\chi_{[u_0 -\varepsilon v > 0]} (u_0 -\varepsilon v) - \chi_{[u_0>0]}u_0\big) \nonumber \\
		&& + \dfrac{f_-(x)}{\varepsilon} \big(\chi_{[u_0 -\varepsilon v \le 0]} (u_0 -\varepsilon v) - \chi_{[u_0\le 0]}u_0\big).
	\end{eqnarray}
	If $x \in [u_0 \le 0]$, since $v \ge 0$ we have $x \in [u_0 - \varepsilon v \le 0]$. Thus, by \eqref{qf+-}
	\begin{eqnarray}\label{f-}
		\dfrac{f(x,u_0 -\varepsilon v)(u_0-\varepsilon v) - f(x,u_0)u_0}{\varepsilon} &=& \dfrac{f_-(x)}{\varepsilon} (u_0-\varepsilon v - u_0) = -f_-(x)v.
	\end{eqnarray}
	Otherwise, if $x \in [u_0 >0]$, just choose $\varepsilon>0$ small enough so that $x \in [u_0 - \varepsilon v >0]$. Then, it follows by \eqref{qf+-} that
	\begin{eqnarray}\label{f+}
		\dfrac{f(x,u_0 -\varepsilon v)(u_0-\varepsilon v) - f(x,u_0)u_0}{\varepsilon} &=& \dfrac{f_+(x)}{\varepsilon} (u_0-\varepsilon v - u_0) = -f_+(x)v.
	\end{eqnarray}
	By combining \eqref{f-} and \eqref{f+} we get
	\begin{eqnarray}
		\dfrac{f(x,u_0 -\varepsilon v)(u_0-\varepsilon v) - f(x,u_0)u_0}{\varepsilon} &=& - f(x,u_0)v,
	\end{eqnarray}
	for $\varepsilon>0$ small enough, which proves the claim. By a similar argument we prove that
	\begin{equation}\label{g0}
		\dfrac{\gamma(x,u_0 -\varepsilon v)- \gamma(x,u_0)}{\varepsilon}  =0.
	\end{equation}
	Now, by combining \eqref{Pfg}, \eqref{g0} and convexity of $\Phi$ we obtain
	$$0 \le - \dfrac{1}{\varepsilon} \int_\Omega \varphi(|\nabla u_0 - \varepsilon \nabla v|)\dfrac{\nabla u_0 - \varepsilon \nabla v}{|\nabla u_0 - \varepsilon \nabla v|} \varepsilon \nabla v dx -  \dfrac{1}{\varepsilon} \int_\Omega\big(f(x,u_0 -\varepsilon v)(u_0-\varepsilon v) - f(x,u_0)u_0\big).$$
	Therefore, taking the limit of $\varepsilon \to 0$, it follows by \eqref{64} that
	$$0 \le - \int_\Omega \varphi(|\nabla u_0|) \dfrac{\nabla u_0}{|\nabla u_0|}\nabla v dx + \int_\Omega f(x,u_0)vdx,$$
	as we desired.
\end{proof}

We are ready to prove Theorem \ref{WV} by establishing the following two propositions.
\begin{prop}\label{weakinfty}
	Let $u_0 \in W_\psi^{1,\Phi}(\Omega)$ be the minimizer of $\mathcal{J}$ obtained in Theorem \ref{existenceintro}. Then, $u_0$ is a weak solution of
	$$-\Delta u - 2 \Delta_\infty u = f_+(x)\exp(-|\nabla u|^2), \; \;  \mbox{in} \; [u > 0],$$
	and
	$$-\Delta u - 2 \Delta_\infty u = f_-(x)\exp(-|\nabla u|^2), \; \;  \mbox{in} \; [u \le 0].$$
\end{prop}

\begin{proof}
	Let us prove the first part. By same reasoning we get the second one. Consider $B \subset [u_0>0]$ and $v\in W_\psi^{1,\Phi}(\Omega)$ such that
	\begin{equation}\label{vweaksol}
		\left\{\begin{array}{rclcl}
			-\mbox{div} \left(\varphi(|\nabla v|)\dfrac{\nabla v}{|\nabla v|}\right) & = & f_+(x), & \mbox{in} & B, \\
			v & = & u_0, & \mbox{in} & B^c.
		\end{array}\right.
	\end{equation}
	In this case, employing Lemma \ref{subsol} and the comparison principle (as outlined in \cite[Lemma 2.4]{AHS}) we deduce that $v \ge u_0$ in $B$. Thus, by minimality of $u_0$ we get
	\begin{eqnarray}\label{0>Pk1}
		0 &\ge& \mathcal{J}(u_0) - \mathcal{J}(v) \nonumber \\
		&=& \int_\Omega \Big(\Phi(|\nabla u_0|) - \Phi(|\nabla v|)\Big) dx - \int_\Omega f_+(x) \Big(\chi_{[u_0>0]} u_0 - \chi_{[v>0]} v \Big)dx \nonumber \\
		&& - \int_\Omega f_-(x) \Big(\chi_{[u_0\le 0]} u_0 - \chi_{[v\le 0]} v \Big)dx + \int_\Omega \gamma_+(x) \Big(\chi_{[u_0 > 0]} - \chi_{[v > 0]} \Big)dx\nonumber \\
		&&+ \int_\Omega \gamma_-(x) \Big(\chi_{[u_0\le 0]} - \chi_{[v\le 0]} \Big)dx \nonumber \\
		&\ge& \int_B \Big(\Phi(|\nabla u_0|) - \Phi(|\nabla v|)\Big) dx - \int_{B} f_+(x)\big(u_0-v\big)dx.
	\end{eqnarray}
	
	\noindent \underline{\bf Claim 1:} $\displaystyle \int_B \Big(\Phi(|\nabla u_0|) - \Phi(|\nabla v|)\Big) dx \ge \int_B |\nabla u_0 - \nabla v|^2dx + \int_B f_+(x)\big(u_0-v\big)dx.$
	
	Indeed, for any $s \in [0,1]$ define $u_s = su_0 + (1-s)v$. By the Fundamental Theorem of Calculus and the mean value we have
	\begin{eqnarray}
		\int_\Omega \Big(\Phi(|\nabla u_0|) - \Phi(|\nabla v|)\Big) dx &=& \sum_{n=1}^\infty \dfrac{1}{n!} \int_\Omega \Big(|\nabla u_0|^{2n} - |\nabla v|^{2n}\Big)dx \nonumber \\
		&=& \sum_{n=1}^\infty \dfrac{1}{n!} \int_{0}^{1}\int_\Omega \Big(|\nabla u_0|^{2n} - |\nabla v|^{2n}\Big)dx\ ds \nonumber \\
		&=& \sum_{n=1}^\infty \dfrac{2n}{n!} \int_{0}^{1}\int_\Omega |\nabla u_s|^{2n-2}\nabla u_s\nabla(u_0-v)dx\ ds. \nonumber
	\end{eqnarray}
	Furthermore, by \eqref{vweaksol} we get
	\begin{eqnarray}
		\int_\Omega \Big(\Phi(|\nabla u_0|) - \Phi(|\nabla v|)\Big) dx &=& \sum_{n=1}^\infty \dfrac{2n}{n!} \int_{0}^{1}\int_\Omega \Big(|\nabla u_s|^{2n-2}\nabla u_s - |\nabla v|^{2n-2}\nabla v\Big)\nabla(u_0-v)dx\ ds \nonumber \\
		&& \hspace{2cm} + \int_\Omega f_+(x)\big(u_0-v\big)dx \nonumber \\
		&=& \sum_{n=1}^\infty \dfrac{2n}{n!} \int_{0}^{1}s^{-1}\int_\Omega \Big(|\nabla u_s|^{2n-2}\nabla u_s - |\nabla v|^{2n-2}\nabla v\Big)\nabla(u_s-v)dx\ ds \nonumber \\
		&& \hspace{2cm} + \int_\Omega f_+(x)\big(u_0-v\big)dx, \nonumber
	\end{eqnarray}
	since $\nabla (u_0-v) = s^{-1} \nabla (u_s-v)$. Now, by Simon's inequality (see \cite[proof of Lemma 2.1]{Si}), namely
	$$\Big(|w_1|^{p-2}w_1 - |w_2|^{p-2}w_2\Big)(w_1-w_2) \ge \dfrac{1}{p2^{p-3}}|w_1-w_2|^p, \; \; p \ge2 \; \mbox{and} \; w_1,w_2 \in \mathbb{R}^N,$$
	with $p=2n$, $w_1=\nabla u_s$ and $w_2=\nabla v$, we obtain
	
	\begin{eqnarray}
		\int_\Omega \Big(\Phi(|\nabla u_0|) - \Phi(|\nabla v|)\Big) dx &\ge& \sum_{n=1}^\infty \dfrac{2n}{n!} \int_{0}^{1} s^{-1}\dfrac{2^{3-2n}}{2n} \int_\Omega |\nabla u_s - \nabla v|^{2n} dx \ ds \nonumber \\
		&&\hspace{2cm} + \int_\Omega f_+(x)\big(u_0-v\big)dx \nonumber \\
		&=& \sum_{n=1}^\infty \dfrac{2^{3-2n}}{n!} \int_{0}^{1} s^{-1} \int_\Omega s^{2n}|\nabla u_0 - \nabla v|^{2n} dx \ ds \nonumber \\
		&&\hspace{2cm} + \int_\Omega f_+(x)\big(u_0-v\big)dx \nonumber \\
		&=& \sum_{n=1}^\infty \dfrac{2^{3-2n}}{n!} \cdot \dfrac{1}{2n} \int_\Omega |\nabla u_0 - \nabla v|^{2n} dx  + \int_\Omega f_+(x)\big(u_0-v\big)dx \nonumber \\
		&\ge& \int_\Omega |\nabla u_0 - \nabla v|^{2} dx + \int_\Omega f_+(x) (u_0-v)dx, \nonumber
	\end{eqnarray}
	as claimed. By combining \eqref{0>Pk1} and Claim 1 we get
	$$\int_B |\nabla u_0 - \nabla v|^2dx \le 0,$$
	which implies that $|\nabla u_0 - \nabla v|=0$, almost everywhere in $B$. Consequently, given that $u_0=v$ on $\partial B$, we conclude that $u_0 = v$ in $B$. Therefore,
	$$-\mbox{div} \left(\varphi(|\nabla u_0|)\dfrac{\nabla u_0}{|\nabla u_0|}\right) = f_+(x), \; \;  \mbox{in} \; [u_0>0],$$
	which leads by Lemma \ref{inftyL} to $-\Delta u_0 - 2 \Delta_\infty u_0 = f_+(x)\exp(-|\nabla u_0|^2)$ in $[u_0 > 0]$, in weak sense, as desired.
\end{proof}

\begin{prop}\label{VS}
	The function $u_0 \in W^{1,\Phi}(\Omega)$ given in Theorem \ref{existenceintro} satisfies the equations
	\begin{equation}\label{viscosity+}
		-\Delta u - 2 \Delta_\infty u = f_+(x)\exp(-|\nabla u|^2), \; \;  \mbox{in} \; [u > 0],
	\end{equation}
	and
	\begin{equation}\label{viscosity-}
		-\Delta u - 2 \Delta_\infty u = f_-(x)\exp(-|\nabla u|^2), \; \;  \mbox{in} \; [u \le 0],
	\end{equation}
	in the viscosity sense.
\end{prop}
\begin{proof}
	We will prove \eqref{viscosity+}, and by same lines we can prove \eqref{viscosity-}. Let $x_0 \in [u_0>0]$ and $\psi$ be a test function such that $u_0(x_0) = \psi(x_0)$ and $u_0-\psi$ has a strict minimum at $x_0$. We must prove that
	$$-\Delta \psi(x_0) - 2\Delta_\infty \psi(x_0) \ge f_+(x)\exp(|\nabla u_0(x_0)|^2).$$
	To do this, suppose by contradiction that the previous inequality is false. In this case, there exists $r>0$ such that
	\begin{equation}\label{NS}
		-\Delta \psi(x) - 2\Delta_\infty \psi(x) < f_+(x)\exp(|\nabla u_0|^2), \quad \mbox{for any} \ x \in B_r(x_0).
	\end{equation}
	Define
	$$m = \inf_{|x-x_0|=r} \big(u-\psi\big)(x) \quad \mbox{and} \quad \widetilde{\psi}(x) = \psi(x) +\frac{m}{2}.$$
	Note that $\widetilde{\psi}(x_0)>u_0(x_0)$. Furthermore, by \eqref{NS} and Lemma \ref{inftyL} we get
	\begin{equation}\label{divpsi}
		-\mbox{div}\left(\varphi(|\nabla \widetilde{\psi}|)\dfrac{\nabla \widetilde{\psi}}{|\nabla \widetilde{\psi}|}\right) < f_+(x), \quad B_r(x_0).
	\end{equation}
	Consider $(\widetilde{\psi} -u_0)^+$ extended by zero outside $B_r(x_0)$. Then, by \eqref{divpsi}
	\begin{equation}\label{<f+}
		\int_{[\widetilde{\psi}>u_0]} \varphi(|\nabla \widetilde{\psi}|)\dfrac{\nabla \widetilde{\psi}}{|\nabla \widetilde{\psi}|} \nabla (\widetilde{\psi}-u_0)dx < \int_{[\widetilde{\psi} > u_0]}f_+(x) (\widetilde{\psi} - u_0)dx.
	\end{equation}
	Moreover,  since $u_0$ is a weak solution of \eqref{viscosity+} (see Theorem \ref{weakinfty}), by Lemma \ref{inftyL} we have 
	\begin{equation}\label{=f+}
		\int_{[\widetilde{\psi}>u_0]} \varphi(|\nabla u_0|)\dfrac{\nabla u_0}{|\nabla u_0|} \nabla (\widetilde{\psi}-u_0)dx = \int_{[\widetilde{\psi} > u_0]}f_+(x) (\widetilde{\psi} - u_0)dx,
	\end{equation}
	using $(\widetilde{\psi} -u_0)^+$ extended by zero outside $B_r(x_0)$ as a test function again.	Now, by similar argument of  Claim 1 of the previous theorem, we get by \eqref{=f+} that
	$$	\int_{[\widetilde{\psi}>u_0]} |\nabla \widetilde{\psi} - \nabla u_0|^2 dx \le \int_{[\widetilde{\psi}>u_0]} \Big(\Phi(|\nabla \widetilde{\psi}|) - \Phi(|\nabla u_0|)\Big)dx -  \int_{[\widetilde{\psi} > u_0]}f_+(x) (\widetilde{\psi} - u_0)dx.$$
	Thus, by convexity of $\Phi$ and \eqref{=f+} we obtain
	\begin{eqnarray}
		\int_{[\widetilde{\psi}>u_0]} |\nabla \widetilde{\psi} - \nabla u_0|^2 dx
		&\le& \int_{[\widetilde{\psi}>u_0]} \left(\varphi(|\nabla \widetilde{\psi}|) \dfrac{\nabla \widetilde{\psi}}{|\nabla \widetilde{\psi}|} - \varphi(|\nabla u_0|) \dfrac{\nabla u_0}{|\nabla u_0|}\right) \nabla (\widetilde{\psi} - u)dx < 0, \nonumber
	\end{eqnarray}
	where in the last inequality we used \eqref{<f+} and \eqref{=f+} again. This leads to a contradiction, proving that $u_0$ is a viscosity supersolution. By analogous reasoning, we establish that $u_0$ is also a viscosity subsolution, thus confirming that $u_0$ is indeed a viscosity solution, thus establishing \eqref{viscosity+}.
\end{proof}

\section{Finite perimeter of the free boundary}\label{SecFP}

In the last section, we are going to show that the free boundaries of the minimizer of the functional $\mathcal{J}$ are locally set of finite perimeter. To achieve this, we introduce an ordered condition on the functions $\gamma_\pm$, expressed by \eqref{gc}, but for convenience we will remember it here as follows:
	$$\left\{\begin{aligned}
		&\gamma(x,s)> c_\gamma, \; \mbox{for any} \ s \neq 0;&\\
		&\gamma(x,0)=0,&
	\end{aligned}\right.$$
for some $c_\gamma >0$. 
As pointed out in \cite{S}, the term $\gamma$ serves to compensate for the phase transition, imposing a flux balance along the free boundary. Consequently, this balance compels the free boundary to acquire a certain regularity.

\begin{proof}[Proof of Theorem \ref{LPF}]
	Consider $B_r$ a ball, such that $B_{2r}\Subset \Omega$. We will prove that the free boundaries have finite perimeter on $B_r$. To do this, let $\eta$ be a cut-off function with $0 \le \eta \le 1$ satisfying
		$$\eta(x) = \left\{\begin{aligned}
			&0,& \text{if}& \; \; x \in B_r;&\\
			&1,& \text{if}& \; \; x \in \overline{\Omega} \setminus B_{2r}.&
		\end{aligned}\right.$$
Fixed $\varepsilon>0$ define
	$$u_\varepsilon \coloneqq (u_0-\varepsilon)^+ - (u_0+\varepsilon)^- \quad \mbox{and} \quad  \widetilde{u}_\varepsilon \coloneqq \eta u_0 + (1-\eta)u_\varepsilon,$$ 
where $v^{\pm}$ denotes the positive and negative parts of $v$. 	Now, for the purpose of using the Coarea formula, we claim that
	\begin{equation}\label{control1}
		\int_{[|u_0|\le\varepsilon]\cap B_r}  |\nabla u_0|^2dx \le C \varepsilon \quad \mbox{and} \quad \mathcal{L}^N\Big([|u_0|\le \varepsilon]\cap B_r\Big) \le C \varepsilon,
	\end{equation}
for some constant $C=C\big(\|f\|_\infty,c_\gamma,\mathcal{L}^N(\Omega)\big)>0$. Indeed, by minimality of $u_0$ on $\mathcal{J}$ we get
	\begin{equation}\label{fp1}
			\int_{B_r} \Big(\Phi(|\nabla u_0|) - \Phi(|\nabla \widetilde{u}_\varepsilon|)\Big)dx + \int_{B_r} \Big(\gamma(x,u_0) - \gamma(x,\widetilde{u}_\varepsilon)\Big)dx \le \int_{B_r} \Big(f(x,u_0)u_0 - f(x,\widetilde{u}_\varepsilon)\widetilde{u}_\varepsilon\Big)dx.
		\end{equation}
Separately, by definition of $\widetilde{u}_\varepsilon$ we have by direct computations
	\begin{eqnarray}\label{fuue}
			\int_{B_r} \Big(f(x,u_0)u_0 - f(x,\widetilde{u}_\varepsilon)\widetilde{u}_\varepsilon\Big)dx 		&\le& \int_{[|u_0|\le \varepsilon]\cap B_r} \Big(f_+(x) + f_-(x)\Big)u_0dx \nonumber \\
			&+& \varepsilon \int_{[u_0> \varepsilon]\cap B_r} f_+(x)dx - \varepsilon \int_{[u_0< - \varepsilon]\cap B_r} f_-(x)dx \nonumber \\
			&\le& 3\varepsilon\|f\|_\infty \mathcal{L}^N(\Omega),
		\end{eqnarray}
as well as, by \eqref{gc}
	\begin{equation}
		\int_{B_r} \Big(\gamma(x,u) - \gamma(x,\widetilde{u}_\varepsilon)\Big)dx	= \int_{[|u_0|\le \varepsilon]\cap B_r} \gamma(x,u_0)dx \ge c_\gamma \mathcal{L}^N\Big([|u_0|\le \varepsilon]\cap B_r\Big).
	\end{equation}
Furthermore, using the definition of $u_\varepsilon$ and $\Phi$, we get
	\begin{equation}\label{66}
		\int_{B_r} \Big(\Phi(|\nabla u_0|) - \Phi(|\nabla \widetilde{u}_\varepsilon|)\Big)dx 
		= \int_{[|u_0|\le\varepsilon]\cap B_r} \Phi(|\nabla u_0|)dx \ge  \int_{[|u_0|\le\varepsilon]\cap B_r} |\nabla u_0|^2dx.
	\end{equation}
Combining \eqref{fp1}--\eqref{66} we obtain
	\begin{eqnarray}
		\int_{[|u_0|\le\varepsilon]\cap B_r} |\nabla u_0|^2 dx + c_\gamma \mathcal{L}^N\Big([|u_0|\le \varepsilon]\cap B_r\Big)
		&\le& \int_{B_r} \Big(\Phi(|\nabla u_0|) - \Phi(|\nabla \widetilde{u}_\varepsilon|)\Big)dx \nonumber \\
		&& + \int_{B_r} \Big(\gamma(x,u_0) - \gamma(x,\widetilde{u}_\varepsilon)\Big)dx \nonumber \\
		&\le& \int_{B_r} \Big(f(x,u_0)u_0 - f(x,\widetilde{u}_\varepsilon)\widetilde{u}_\varepsilon\Big)dx \nonumber \\
		&\le& 3\varepsilon \|f\|_\infty\mathcal{L}^N(\Omega), \nonumber
	\end{eqnarray}
and the claim follows.
	
Now, consider the set $A_\varepsilon \coloneqq [0<u_0\le\varepsilon]\cap B_r$. In this case, by H\"{o}lder's inequality and \eqref{control1} we get
	$$\int_{[0<u_0<\varepsilon]\cap B_r} |\nabla u_0|dx \le \int_{A_\varepsilon} |\nabla u_0|dx \le \left(\int_{A_\varepsilon} |\nabla u_0|^2dx\right)^{\frac{1}{2}} \mathcal{L}^N(A_\varepsilon)^{\frac{1}{2}} \le C \varepsilon.$$
By Coarea formula
	$$\int_{0}^{\varepsilon} \mathcal{H}^{N-1}\Big(\partial^* \{u_0>t\}\cap B_r\Big)dt = \int_{[0<u_0<\varepsilon]\cap B_r} |\nabla u_0|dx \le C\varepsilon.$$
Taking $\varepsilon = 1/n$ we get that there is $\delta_n \in [0,1/n]$ such that
	$$\mathcal{H}^{N-1}\Big(\partial^* \{u_0>\delta_n\}\cap B_r\Big) \le n \int_{0}^{1/n} \mathcal{H}^{N-1}\Big(\partial^* \{u_0>t\}\cap B_r\Big)dt \le C.$$
Since $[u_0>\delta_n] \to [u_0>0]$ in $L^1(\Omega)$ and the perimeter is lower semicontinuous with respect to $L^1$-convergence of sets \cite[see Proposition 12.15]{M}, we get
	$$\mathcal{H}^{N-1}\Big(\partial^* \{u_0 > 0\}\cap B_r\Big) \le C,$$
where $C=C(\|f\|_\infty,c_\gamma,\mathcal{L}^N(\Omega))>0$, as desired. The proof for the reduced boundary $\partial^* \{u_0 < 0\}$ is similar, considering the set $B_\varepsilon = [-\varepsilon < u_0 < 0]\cap B_r$. Finally, these estimates imply that the reduced free boundaries $\partial^*\{u_0>0\}$ and $\partial^*\{u_0<0\}$ are $(N-1)$-rectifiable sets. Therefore, their respective Hausdorff dimensions are at most $N-1$, and this complete the proof of theorem.
\end{proof}

\begin{remark}
	It is noteworthy that the Theorem \ref{LPF} can be generalized to a general function $G$ satisfying \eqref{Lie}. Indeed, note that the only influence of the operator lies in the verification of the inequality \eqref{66}, which remains valid when considering the general case, where we can use $(G_1)$.
\end{remark}


\begin{thebibliography}{99}
	\small
	
	\bibitem{A} R.A. Adams, \emph{Sobolev Spaces}, Academic Press, New York, 1975.
	
	\bibitem{AHS} {C.O. Alves; A.R.F de Holanda;  J.A. Santos,}	\emph{Existence of positive solutions for a class of semipositone quasilinear problems through {O}rlicz-{S}obolev space}. {Proc. Amer. Math. Soc.} {147},  {285-299} {(2019)}.
	
	\bibitem{AT} {M.D. Amaral; E.T. Teixeira}, \emph{Free transmission problems}, {Comm. Math. Phys.}, {337}, {3}, {1465--1489}, (2015).
	
	\bibitem{AMW} {D.M. Anderson;  G.B McFadden; A.A Wheeler}, \emph{Diffuse-interface methods in fluid mechanics}, {Annu. Rev. Fluid Mech.}, {30}, {139--165}, {(1998)}.
	
	\bibitem{ACM} J.C. Arciero; P. Causin; F. Malgaroli, \emph{Mathematical methods for modeling the microcirculation}, AIMS Biophysics, 4(3), 362--399, (2017).
	
	\bibitem{ADS} {S.N. Antontsev; J.I. D\'iaz; S. Shmarev}, \emph{Energy methods for free boundary problems: {Applications to nonlinear PDEs and fluid mechanics},}, {Progress in Nonlinear Differential Equations and their
		Applications}, {48}, {Birkh\"auser Boston, Inc., Boston, MA}, {xii+329}, {(2002)}.
		
	\bibitem{AB} J. Azzam; J. Bedrossian, \emph{Bounded mean oscillation and the uniqueness of active scalar equations.} Trans. Am. Math. Soc. 367, 3095–3118 (2015).
	
	\bibitem{BB} J. Bear; Y. Bachmat, \emph{Introduction to Modeling of Transport Phenomena
		in Porous Media}, Springer Science \& Business Media. (1990).
	
	\bibitem{BV} J. Bear; A. Verruijt, \emph{Modeling Groundwater Flow and Pollution}, Springer Science \& Business Media. (1987).
	
	\bibitem{BM} {M. Bocea; M. Mih\u{a}ilescu}, \emph{On the existence and uniqueness of exponentially harmonic maps and some related problems}, {Israel J. Math.}, {230},  {2}, {795--812}, {(2019)}.
	
	\bibitem{B} {J.E.M. Braga}, \emph{On the {L}ipschitz regularity and asymptotic behaviour of the free boundary for classes of minima of inhomogeneous two-phase 		{A}lt-{C}affarelli functionals in {O}rlicz spaces}, {Ann. Mat. Pura Appl. (4)}, {197}, {(2018)}.
	
	\bibitem{BM} {J.E.M Braga; D.R. Moreira}, \emph{Uniform {L}ipschitz regularity for classes of minimizers in two phase free boundary problems in {O}rlicz spaces with small density on the negative phase}, {Ann. Inst. H. Poincar\'{e} C Anal. Non Lin\'{e}aire}, {31}, {4}, {823--850}, {(2014)}.
	
	\bibitem{CSS} {N. Cantizano; A. Salort; J. Spedaletti}, \emph{Continuity of solutions for the {$\Delta_{\phi}$}-{L}aplacian	operator}, {Proc. Roy. Soc. Edinburgh Sect. A}, {151}, {4}, {1355--1382}, {(2021)}.
	
	\bibitem{CC} {R.S. Cantrell; C. Cosner}, \emph{Spatial ecology via reaction-diffusion equations}, {Wiley Series in Mathematical and Computational Biology}, {John Wiley \& Sons, Ltd., Chichester}, {xvi+411}, (2003).
	
	\bibitem{DE} D.M Duc; J. Eells, \emph{Regularity of exponetional harmonic functions}, Int. J.	Math. 2, 395--408 (1991).
	
	\bibitem{FIN}
	{N. Fukagai; M. Ito; K. Narukawa,} \emph{Positive solutions of quasilinear elliptic equations with critical {O}rlicz-{S}obolev nonlinearity on $\mathbb{R}^N$},	{Funkcial. Ekvac.}	{49},   {235--267}	{(2006)}.
	
	\bibitem{H} {A. Hastings}, \emph{Population Biology: Concepts and Models}, {Springer New York, NY}, {xvi+220}, (1997).
	
	\bibitem{LU} O.A. Ladyzhenskaya; N.N. Ural’tseva, \emph{Linear and Quasilinear Elliptic Equations.} Mathematics in Science	 and Engineering, vol. 46. Academic Press, New York (1968).
	
	\bibitem{LQO} {R. Leit\~{a}o, O.S. de Queiroz, E.V. Teixeira;} \emph{Regularity for degenerate two-phase free boundary problems}, {Ann. Inst. H. Poincar\'{e} C Anal. Non Lin\'{e}aire}, {32}, {4}, {741--762}, {(2015)}.
	
	
	\bibitem{L} G.M. Lieberman, \emph{The natural generalization of the natural conditions of Ladyzhenskaya and Ural’tseva for elliptic equations.} Comm. Partial Differential Equations 16, no. 2-3, 311--361, (1991).
	
	\bibitem{L2} G.M. Lieberman \emph{Boundary regularity for solutions of degenerate elliptic equations,} Nonlinear Anal. 12, 1203--1219,  (1988).
	
	\bibitem{M} {F. Maggi,} \emph{Sets of finite perimeter and geometric variational problems}, {Cambridge Studies in Advanced Mathematics}, {135}, {An introduction to geometric measure theory}, {Cambridge University Press, Cambridge}, {2012}, {xx+454}.
	
	\bibitem{MW} {S. Mart\'{\i}nez; N. Wolanski}, \emph{A minimum problem with free boundary in {O}rlicz spaces}, {Adv. Math.}, {218}, {6}, {1914--1971}, {(2008)}.
	
	\bibitem{N} H. Naito, \emph{On a local H\"{o}lder continuity for a minimizer of the exponential energy functinal}, {Nagoya Math. J.}, 129, 97--113, (1993).
	
	\bibitem{RaoRen} M.M. Rao; Z.D. Ren, \emph{Theory of Orlicz spaces}. Monographs and Textbooks in Pure and Applied Mathematics 146, Marcel Dekker, Inc., New York (1991).
	
	\bibitem{SS} {J.A. Santos; S.H.M. Soares}, \emph{A limiting free boundary problem for a degenerate operator in {O}rlicz-{S}obolev spaces}, {Rev. Mat. Iberoam.}, {36}, {6}, {1687--1720}, {(2020)}.	
	
	\bibitem{SS2}{J.A. Santos; S.H.M Soares}, \emph{Optimal design problems for a degenerate operator in	{O}rlicz-{S}obolev spaces}, {Calc. Var. Partial Differential Equations}, {59}, {6}, {Paper No. 183, 23}, {(2020)}.
	
	\bibitem{SP} M. Sharan; A.S. Popel, \emph{A two-phase model for flow of blood in narrow tubes with increased effective viscosity near the wall}, Biorheology, 38(5-6): p. 415-28. (2001).
	
	\bibitem{S} {H. Shrivastava,} \emph{A non-isotropic free transmission problem governed by quasi-linear operators}, {Ann. Mat. Pura Appl. (4)}, {200}, {6}, {2455--2471}, {(2021)}.
	
	\bibitem{Si} J. Simon, \emph{R\'{e}gularit\'{e} de la solution d'une \'{e}quation non	lin\'{e}aire dans {${\bf R}\sp{N}$}}, {Journ\'{e}es d'{A}nalyse {N}on {L}in\'{e}aire ({P}roc.	{C}onf., {B}esan\c{c}on, 1977)}, 205--227. Lecture Notes in Math. 665, Springer, Berlin, (1978).
	
	\bibitem{ZZ} {C. Zhang; S. Zhou}, \emph{On a class of non-uniformly elliptic equations}, {NoDEA Nonlinear Differential Equations Appl.}, {19}, {3}, {345--363}, {(2012)}.
	
	\bibitem{ZT} {J. Zheng; L.S. Tavares}, \emph{A free boundary problem with subcritical exponents in {O}rlicz spaces}, {Ann. Mat. Pura Appl. (4)}, {201}, {2}, {695--731}, {(2022)}.
	
	\bibitem{ZZZ} {J. Zheng; Z. Zhang; P. Zhao}, \emph{A minimum problem with two-phase free boundary in {O}rlicz spaces}, {Monatsh. Math.}, {172}, {3-4}, {441--475}, {(2013)}.
\end{thebibliography}
\end{document}